\documentclass[12pt,reqno]{amsart}

\usepackage{amsmath , amssymb, latexsym }
\usepackage{graphics}
\usepackage{graphicx}

\usepackage[mathscr]{eucal}

\newdimen\unit\newdimen\psep\newcount\nd\newcount\ndx\newbox\dotb\newbox\ptbox
\newdimen\dx\newdimen\dy\newdimen\dxx\newdimen\dyy\newdimen\hgt
\newdimen\xoff\newdimen\yoff
\newcommand\clap[1]{\hbox to 0pt{\hss{#1}\hss}}
\newcommand\vdisk[1]{{\font\dotf=cmr10 scaled #1\dotf.}}
\newcommand\varline[2]{\setbox\dotb\hbox{\vdisk{#1}}\xoff=-.5\wd\dotb
\wd\dotb=0pt\yoff=-.5\ht\dotb\psep=#2\ht\dotb}
\newcommand\varpt[1]{\setbox\ptbox\clap{\vdisk{#1}}\setbox\ptbox
\hbox{\raise-.5\ht\ptbox\box\ptbox}}
\newcommand\cpt{\copy\ptbox}
\newcommand\point[3]{\rlap{\kern#1\unit\raise#2\unit\hbox{#3}}}
\newcommand\setnd[4]{\dx=#3\unit\advance\dx-#1\unit\divide\dx by\psep
\dy=#4\unit\advance\dy-#2\unit\divide\dy by\psep \multiply\dx
by\dx\multiply\dy by\dy\advance\dx\dy\nd=1\advance\dx-1sp
\loop\ifnum\dx>0\advance\dx-\nd sp\advance\nd1\advance\dx-\nd
sp\repeat}
\newcommand\dl[4]{{\setnd{#1}{#2}{#3}{#4}\dline{#1}{#2}{#3}{#4}\nd}}
\newcommand\dline[5]{{\nd=#5\hgt=#2\unit\dx=#3\unit\advance\dx-#1\unit
\divide\dx by\nd\dy=#4\unit\advance\dy-#2\unit\divide\dy by\nd
\advance\hgt\yoff\rlap{\kern#1\unit\kern\xoff\loop\ifnum\nd>1\advance\nd-1
\advance\hgt\dy\kern\dx\raise\hgt\copy\dotb\repeat}}}

\newcommand\qellip[4]{{\setnd{0}{0}{#3}{#4}\dx=\unit\dy=0pt\raise\yoff\rlap{%
\kern#1\unit\kern\xoff\raise#2\unit\hbox{\loop\ifnum\dx>0\rlap{\kern#3\dx
\raise#4\dy\copy\dotb}\hgt=\dx\divide\hgt
by\nd\advance\dy\hgt\hgt=\dy \divide\hgt
by\nd\advance\dx-\hgt\repeat\rlap{\raise#4\dy\copy\dotb}}}}}
\newcommand\bez[6]{{\setnd{#1}{#2}{#3}{#4}\ndx=\nd\setnd{#3}{#4}{#5}{#6}
\ifnum\ndx>\nd\nd=\ndx\fi\dx=#3\unit\advance\dx-#1\unit\dy=#4\unit
\advance\dy-#2\unit\dxx=#5\unit\advance\dxx-#1\unit\dyy=#6\unit\advance
\dyy-#2\unit\advance\dxx-2\dx\advance\dyy-2\dy\divide\dxx
by\nd\divide\dyy
by\nd\advance\dx.25\dxx\advance\dy.25\dyy\divide\dx
by\nd\divide\dy by\nd \multiply\nd
by2\dx=100\dx\dy=100\dy\dxx=100\dxx\dyy=100\dyy\divide\dxx by\nd
\divide\dyy
by\nd\hgt=#2\unit\raise\yoff\rlap{\kern#1\unit\kern\xoff
\raise\hgt\copy\dotb\loop\ifnum\nd>0\advance\nd-1\advance\hgt0.01\dy
\kern0.01\dx\raise\hgt\copy\dotb\advance\dx\dxx\advance\dy\dyy\repeat}}}
\newcommand\ptu[3]{\point{#1}{#2}{\cpt\raise1ex\clap{$\scriptstyle{#3}$}}}
\newcommand\ptd[3]{\point{#1}{#2}{\cpt\raise-1.8ex\clap{$\scriptstyle{#3}$}}}
\newcommand\ptr[3]{\point{#1}{#2}{\cpt\raise-.4ex\rlap{$\ \scriptstyle{#3}$}}}
\newcommand\ptl[3]{\point{#1}{#2}{\cpt\raise-.4ex\llap{$\scriptstyle{#3}\ $}}}
\newcommand\ptlu[3]{\point{#1}{#2}{\raise.8ex\clap{$\scriptstyle{#3}$}}}
\newcommand\ptld[3]{\point{#1}{#2}{\raise-1.6ex\clap{$\scriptstyle{#3}$}}}
\newcommand\ptlr[3]{\point{#1}{#2}{\raise-.4ex\rlap{$\,\scriptstyle{#3}$}}}
\newcommand\ptll[3]{\point{#1}{#2}{\raise-.4ex\llap{$\scriptstyle{#3}\,$}}}
\newcommand\pt[2]{\point{#1}{#2}{\cpt}}

\newcommand\medline{\varline{800}{.5}}
\newcommand\thnline{\varline{400}{.6}}

\varpt{2500}\thnline\unit=2em

\newtheorem{thm}{Theorem}

\newtheorem*{vBKlemma}{The van den Berg--Kesten Lemma}

\newtheorem{conj}{Conjecture}

\newtheorem{prob}{Problem}
\newtheorem{lemma}[thm]{Lemma}
\newtheorem{prop}[thm]{Proposition}
\newtheorem{cor}[thm]{Corollary}

\newtheorem{obs}[thm]{Observation}
\theoremstyle{definition}\newtheorem{rmk}{Remark}
\theoremstyle{definition}\newtheorem*{defn}{Definition}
\theoremstyle{definition}
\theoremstyle{definition}
\newcommand{\ds}{\displaystyle}
\newcommand{\ul}{\underline}

\makeatletter
\def\Ddots{\mathinner{\mkern1mu\raise\p@
\vbox{\kern7\p@\hbox{.}}\mkern2mu
\raise4\p@\hbox{.}\mkern2mu\raise7\p@\hbox{.}\mkern1mu}}
\makeatother

\def\C{\mathcal{C}}

\def\E{\mathcal{E}}

\def\HH{\mathcal{H}}

\def\J{\mathcal{J}}

\def\M{\mathcal{M}}
\def\n{\mathcal{N}}

\def\P{\mathcal{P}}

\def\X{\mathcal{X}}

\def\Ex{\mathbb{E}}

\def\N{\mathbb{N}}
\def\Pr{\mathbb{P}}

\def\RR{\mathbb{R}}
\def\ZZ{\mathbb{Z}}

\def\le{\leqslant}
\def\ge{\geqslant}

\def\eps{\varepsilon}
\def\->{\rightarrow}
\def\<{\langle}
\def\>{\rangle}
\def\diam{\textup{diam}}

\def\grr{\textup{girth}}
\def\lg{\textup{long}}
\def\sht{\textup{short}}
\def\Bin{\textup{Bin}}
\def\dz{\textup{d}z}
\def\x{\mathbf{x}}
\def\y{\mathbf{y}}
\def\z{\mathbf{z}}
\def\a{\mathbf{a}}
\def\b{\mathbf{b}}
\def\c{\mathbf{c}}
\def\d{\mathbf{d}}
\def\k{\mathbf{k}}
\def\m{\mathbf{m}}
\def\r{\mathbf{r}}
\def\s{\mathbf{s}}

\def\v{\mathbf{v}}
\def\0{\mathbf{0}}
\def\1{\mathbf{1}}

\begin{document}

\title{The sharp threshold for bootstrap percolation in all dimensions}

\author{J\'ozsef Balogh}
\address{Department of Mathematics\\ University of Illinois\\ 1409 W. Green Street\\ Urbana, IL 61801\\ and\\ Department of Mathematics\\ University of California\\ San Diego, La Jolla, CA 92093}\email{jobal@math.uiuc.edu}

\author{B\'ela Bollob\'as}
\address{Trinity College\\ Cambridge CB2 1TQ\\ England\\ and \\ Department of Mathematical Sciences\\ The University of Memphis\\ Memphis, TN 38152, USA} \email{B.Bollobas@dpmms.cam.ac.uk}

\author{Hugo Duminil-Copin}
\address{D\'epartement de Math\'ematiques, Universit\'e de Gen\`eve, 2-4 rue du Li\`evre, 1211 Genv\`eve, Suisse} \email{hugo.duminil@unige.ch}

\author{Robert Morris} 
\address{IMPA, Estrada Dona Castorina 110, Jardim Bot\^anico, Rio de Janeiro, RJ, Brasil} \email{rob@impa.br}

\thanks{\emph{AMS 2010 subject classifications}: 60K35, 60C05}
\thanks{The first author was supported by NSF CAREER Grant DMS-0745185, UIUC Campus Research Board Grants 09072 and 08086, and OTKA Grant K76099. The second author was supported by NSF grants CNS-0721983, CCF-0728928 and DMS-0906634, and ARO grant W911NF-06-1-0076. The third was supported by ANR grant BLAN-3-134462 and the Swiss NSF, and the fourth by MCT grant PCI EV-8C, ERC Advanced grant DMMCA, and a Research Fellowship from Murray Edwards College, Cambridge}
\keywords{Bootstrap percolation, sharp threshold}

\begin{abstract}
In $r$-neighbour bootstrap percolation on a graph $G$, a (typically random) set $A$ of initially `infected' vertices spreads by infecting (at each time step) vertices with at least $r$ already-infected neighbours. This process may be viewed as a monotone version of the Glauber dynamics of the Ising model, and has been extensively studied on the $d$-dimensional grid $[n]^d$. The elements of the set $A$ are usually chosen independently, with some density $p$, and the main question is to determine $p_c([n]^d,r)$, the density at which percolation (infection of the entire vertex set) becomes likely.  

In this paper we prove, for every pair $d,r \in \N$ with $d \ge r \ge 2$, that
$$p_c\big( [n]^d,r \big) \; = \; \left( \frac{\lambda(d,r) + o(1)}{\log_{(r-1)} (n)} \right)^{d-r+1}$$
as $n \to \infty$, for some constant $\lambda(d,r) > 0$, and thus prove the existence of a sharp threshold for percolation in any (fixed) number of dimensions. We moreover determine $\lambda(d,r)$ for every $d \ge r \ge 2$.
\end{abstract}

\maketitle

\section{Introduction}\label{intro}

Cellular automata, which were introduced by von Neumann (see~\cite{vN}) after a suggestion of Ulam~\cite{Ulam}, are dynamical systems (defined on a graph $G$) whose update rule is homogeneous and local. In this paper we shall study a particular cellular automaton, known as $r$-neighbour bootstrap percolation, which may be thought of as a monotone version of the Glauber dynamics of the Ising model of ferromagnetism. We shall prove the existence of a sharp threshold for percolation in the $r$-neighbour model on the grid $[n]^d$, where $d \ge r \ge 2$ are fixed and $n \to \infty$, and moreover we shall determine the critical probability $p_c([n]^d,r)$ up to a factor of $1 + o(1)$. Our main theorem settles the major open question in bootstrap percolation. 

Given a (finite or infinite) graph $G$, and an integer $r \in \N = \{0,1,2,\ldots\}$, the $r$-neighbour bootstrap process on $G$ is defined as follows. Let $A$ be a set of initially `infected' vertices. At each time step, infect all of the vertices which have at least $r$ already-infected neighbours. To be precise, let $A_0 = A$, and define
$$A_{t+1} \; := \; A_t \cup \big\{v \in V(G) : |N(v) \cap A_t| \ge r \big\}$$
for each $t \in \N$, where $N(v)$ denotes the set of (nearest) neighbours of $v$ in $G$, and $|S|$ denotes the cardinality of a set $S$. We think of the set $A_t$ as the vertices which are infected at time $t$, and write $[A] = \bigcup_t A_t$ for the closure of $A$ under the process. We say that the set $A$ \emph{percolates} if the entire vertex set is eventually infected, i.e., if $[A] = V(G)$. 

The bootstrap process was introduced in 1979 by Chalupa, Leath and Reich~\cite{CLR} in the context of disordered magnetic systems, and has been studied extensively by mathematicians (see, for example,~\cite{AL,BB,BPP,CC,Hol,Sch}) and physicists~\cite{ALev,BQdS,GLBD,KGC}, as well as by computer scientists~\cite{DR,FELPS} and sociologists~\cite{Gran,Watts}, amongst others. Motivated by these physical models, we shall consider bootstrap percolation on the grid $[n]^d = \{1,\ldots,n\}^d$, and an initial set $A \subset V(G)$ whose elements are chosen independently at random, each with probability $p$. We shall write $\Pr_p$ for this distribution; throughout the paper, $A$ will always denote a random subset of $V(G)$ chosen according to $\Pr_p$. 

It is clear that the probability of percolation is increasing in $p$, and so we may define the critical probability, $p_c(G,r)$ as follows:
$$p_c(G,r) \; := \; \inf \Big\{ p \,:\, \Pr_p\big(A \textup{ percolates in the $r$-neighbour process on $G$}\big) \ge 1/2 \Big\}.$$
Our aim is to give sharp bounds on $p_c([n]^d,r)$, and to bound the size of the `critical window' in which the probability of percolation shifts from $o(1)$ to $1 - o(1)$. 

The first rigorous results on bootstrap percolation were obtained by van Enter~\cite{vE} and Schonmann~\cite{Sch}, on the infinite lattice $\ZZ^d$, and by Aizenman and Lebowitz~\cite{AL}, on the finite grid. In particular, Schonmann proved that $p_c(\ZZ^d,r) = 0$ if $r \le d$, and $p_c(\ZZ^d,r) = 1$ otherwise. The finite-volume behaviour (also known as `metastability') was studied in~\cite{AL,CC,CM}, and the threshold function $p_c([n]^d,r)$ was determined up to a constant factor, for all $d \ge r \ge 2$, by Cerf and Manzo~\cite{CM}. The first sharp threshold was determined by Holroyd~\cite{Hol}, in the case $d=r=2$, who proved that
$$p_c\big( [n]^2,2 \big) \; = \; \frac{\pi^2}{18\log n} \,+\, o\left( \frac{1}{\log n} \right)$$
as $n \to \infty$, and a corresponding result in three dimensions was recently proved in~\cite{d=r=3}. However, a longstanding open question (see, for example,~\cite{AL,BBsharp,CM,Hol}) was to determine whether there is sharp transition for $p_c([n]^d,r)$ (for fixed $d$ and $r$, as $n \to \infty$), and if so, whether there is a limiting constant. We resolve this question affirmatively, and determine the constant for every pair $(d,r)$.

In order to state our main result we first need to recall some functions from~\cite{d=r=3}. First, for each $k \in \N$, let
\begin{equation} \label{betadef} 
\beta_k(u) \; := \; \frac{1}{2}\left( 1 \, - \, (1-u)^k \, + \, \Big( 1 + (4u-2)(1-u)^k + (1-u)^{2k} \Big)^{1/2} \right),\nonumber
\end{equation}
so $\beta_k(u)^2 = \big(1 - (1-u)^k\big) \beta_k(u) + u(1-u)^k$, and let
\begin{equation} \label{gdef} 
g_k(z) \; := \; - \log\Big( \beta_k \big( 1 - e^{-z} \big) \Big).\nonumber
\end{equation}
Now, for each $2 \le r \le d \in \N$, let
\begin{equation} \label{lamdef}
\lambda(d,r) \; := \; \int_0^\infty g_{r-1}(z^{d-r+1}) \, \dz.
\end{equation}

The following theorem is the main result of this paper. Let $\log$ denote the natural logarithm, and let $\log_{(r)}$ denote an $r$-times iterated logarithm, $\log_{(r+1)}(n) = \log \big( \log_{(r)} (n) \big)$. 

\begin{thm}\label{sharp}
Let $d,r \in \N$, with $d \ge r \ge 2$. Then
$$p_c([n]^d,r) \; = \; \left( \frac{\lambda(d,r) + o(1)}{\log_{(r-1)} (n)} \right)^{d-r+1}$$
as $n \to \infty$. 
\end{thm}

\begin{rmk}
We shall moreover obtain explicit bounds on the probability that $A$ percolates outside the critical window. To be precise, for any $\eps > 0$ we shall prove that, if $p = (1-\eps)p_c$, then
$$\Pr_p\big( A \textup{ percolates} \big) \; \le \; n^{- d(r-2) - \delta}$$
for some $\delta = \delta(\eps) > 0$ (see Corollary~\ref{2r2cor} and  Theorem~\ref{genthm}). In~\cite{d=r=3} it was proved that if $p = (1+\eps)p_c$, then 
$$\Pr_p\big( A \textup{ percolates} \big) \; \ge \; 1 \,-\, \exp\left( - \frac{ n^{d-1} } { (\log n)^{3d} } \right).$$
\end{rmk}

Some special cases of Theorem~\ref{sharp} were known previously. Indeed, as noted above, the case $d=r=2$ was proved by Holroyd~\cite{Hol}, and the case $d = r = 3$, and the upper bound in Theorem~\ref{sharp}, were proved by Balogh, Bollob\'as and Morris~\cite{d=r=3}. Holroyd~\cite{Holddim} also proved a sharp threshold for a `modified' bootstrap percolation in an arbitrary (constant) number of dimensions. The modified model is much simpler to study, however, and the critical threshold differs from ours by a factor of about $d$. A weaker notion of sharpness was proven for $r = 2$ and all $d \ge 2$ by Balogh and Bollob\'as~\cite{BBsharp}, using a general result of Friedgut and Kalai~\cite{FK}. Their result implies that the critical window is of order $o(p_c)$, but not that the sequence $p_c([n]^d,2) (\log n)^{d-1}$ converges.

Although we cannot solve the integral~\eqref{lamdef} exactly, it is not too hard to prove that the function $\lambda(d,r)$ has some nice properties. In particular, $\lambda(d,r) < \infty$ for every $d \ge r \ge 2$, $\lambda(2,2) = \ds\frac{\pi^2}{18}$ (see~\cite{Hol} and also~\cite{HLR}), $\lambda(d,2) = \ds\frac{d-1}{2} + o(1)$, and 
$$\lambda(d,d) \; = \; \left( \frac{\pi^2}{6} + o(1) \right) \frac{1}{d}$$ as $d \to \infty$ (see~\cite{d=r=3}). 

The following table lists some approximate values of $\lambda(d,r)$ for $2 \le d \le 7$:

\pagebreak

\begin{center}
$d$
\end{center}
\smallskip
\begin{center}
\begin{tabular}{cc||c|c|c|c|c|c}
&& 2 & 3 & 4 & 5 & 6 & 7  \\ \hline \hline
& 2 \; & \; 0.5483 \; & \; 0.9924 \; & \; 1.4797 \; & \; 1.9764 \; & \; 2.4760 \; & \; 2.9768 \; \\
& 3 \; & - & 0.4039 & 0.8810 & 1.3864 & 1.8961 & 2.4078 \\
$r$ \; & 4 \; & - & - & 0.3198 & 0.8024 & 1.3162 & 1.8338 \\
& 5 \; & - & - & - & 0.2650 & 0.7431 & 1.2606 \\
& 6 \; & - & - & - & - & 0.2265 & 0.6963 \\
& 7 \; & - & - & - & - & - & 0.1979 %
\end{tabular}\\\ \\[+1ex]
\smallskip
$\lambda(d,r)$
\end{center}
\medskip

We remark finally that the bootstrap process has also been studied on several other graphs, such as high dimensional tori~\cite{BB,Maj,n^d}, infinite trees~\cite{BPP,BS,FS}, the random regular graph~\cite{BP,Svante}, `locally tree-like' regular graphs~\cite{Maj}, and the Erd\H{o}s-R\'enyi random graph $G_{n,p}$~\cite{JLTV}. Some of the techniques from these papers (and those mentioned earlier) have been used to prove results about the low-temperature Glauber dynamics of the Ising model~\cite{CM2,FSS,Me}. Some very recent results on bootstrap percolation in two dimensions can be found in~\cite{DCH,GHM}, see Section~\ref{Qsec} for more details.

We shall prove Theorem~\ref{sharp} by induction on $r$, and in order for the proof to work we shall need to strengthen the induction hypothesis. A \emph{bootstrap structure} is a graph $G$, together with a function $r: V(G) \to \N$ which assigns a `threshold' to each vertex of $G$. Bootstrap percolation on such a structure is then defined in the obvious way, by setting $A_0 = A$ and
$$A_{t+1} \; := \; A_t \: \cup \: \big\{ v \in V(G) \: : \: |N(v) \cap A_t| \ge r(v) \big\}$$
for each $t \ge 0$, and letting $[A] = \bigcup_t A_t$. 

The following family of bootstrap structures, which we call $C([n]^d \times [k]^\ell,r)$, will be a crucial tool in our proof. We think of $[n]^d \times [k]^\ell$ as a box $[n]^d$ of `thickness' $[k]^\ell$.
\begin{defn}
Let $n,d,k,\ell,r \in \N$. Then $C([n]^d \times [k]^\ell,r)$ is the bootstrap structure such that
\begin{enumerate}
\item[$(a)$] the vertex set is $[n]^d \times [k]^\ell$,\\[-2ex]
\item[$(b)$] the edge set is induced by $\ZZ^{d+\ell}$,\\[-2ex]
\item[$(c)$] $v = (a_1,\ldots,a_d,b_1,\ldots,b_\ell)$ has threshold $r + |\{j \in [\ell] : b_j \not\in \{1,k\}\}|$.
\end{enumerate}
\end{defn}
\noindent Let $B([n]^d,r)$ denote the bootstrap structure on $[n]^d$ in which every vertex has threshold $r \in \N$, and note that $B([n]^d,r) = C([n]^d \times [k]^0,r)$.

We shall in fact determine a sharp threshold for percolation on $C([n]^d \times [k]^\ell,r)$ for every $d \ge r \ge 2$ and every $\ell \in \N$, when $k = k(n) \to \infty$ sufficiently slowly (see Theorem~\ref{genthm}, below, and Theorem~5 of~\cite{d=r=3}). The main difficulty will lie in proving the result below, which implies the lower bound in the case $r = 2$. We define the \emph{diameter} $\diam(S)$ of a set $S \subset \ZZ^{d+\ell}$ to be
$$\diam(S) \; := \; \sup \Big\{ \| x - y \|_\infty + 1 : x,y \in S \textup{ and } (x \leftrightarrow y)_S \Big\},$$
where we write $(x \leftrightarrow y)_S$ to indicate that there exists a path from $x$ to $y$ (in the graph $\ZZ^{d+\ell}$) using only vertices of $S$. 

Recall that $[A]$ denotes the closure of $A$ under the bootstrap process. The following theorem will be the base case of our proof by induction. 

\begin{thm}\label{r=2}
Let $d,\ell \in \N$, with $d \ge 2$, and let $\eps > 0$. Let $B > 0$ and $k \ge k_0(B) \in \N$ be sufficiently large, and let the elements of $A \subset C([n]^d \times [k]^\ell,2)$ be chosen independently at random with probability $p$, where
$$p \; = \; p(n) \; \le \; \left( \ds\frac{\lambda(d+\ell,\ell+2) - \eps}{\log n} \right)^{d-1}.$$
Then
 $$\Pr_p\big( \diam([A]) \ge B \log n \big) \; \to \; 0$$
 as $n \to \infty$.
 \end{thm}

The rest of the paper is organised as follows. In Sections~\ref{tools} and~\ref{hiersec} we review some basic definitions and tools from~\cite{d=r=3}, and in Section~\ref{sketchsec} we give a brief sketch of the proof. In Section~\ref{crosssec} we bound the probability that a rectangle is `crossed' by $A$, in Section~\ref{analsec} we present some basic analytic tools, and in Section~\ref{T2sec} we deduce Theorem~\ref{r=2} using ideas from Holroyd's proof of the case $d=r=2$. In Section~\ref{CCsec} we recall the method of Cerf and Cirillo~\cite{CC}, and in Section~\ref{T1PfSec} we deduce Theorem~\ref{sharp}. Finally, in Section~\ref{Qsec}, we state some open problems and conjectures.

\section{Tools and definitions}\label{tools}

In this section we recall various tools and definitions from~\cite{d=r=3} which we shall use throughout the paper. Define a \emph{rectangle} $R$ in $[n]^d \times [k]^\ell = \{1,\ldots,n\}^d \times \{1,\ldots,k\}^\ell$ to be a set
$$\big[ (a_1,\ldots,a_d),(b_1,\ldots,b_d) \big] \; := \; \big\{ (x_1,\ldots,x_d,y_1,\ldots,y_\ell) \,:\, x_i \in [a_i,b_i], y_i \in [k] \big\},$$
where $[a,b] = \{a,a+1,\ldots,b\}$ and $[b] = [1,b]$. We also identify these with rectangles in $[n]^d = [n]^d \times [k]^0$ in the obvious way. The \emph{dimensions} of $R$ is the vector
$$\dim(R) \; := \; (b_1 - a_1 + 1, \ldots, b_d - a_d + 1) \in \N^d$$
and the \emph{semi-perimeter} of $R$ is
$$\phi(R) \; := \; \sum_i \big( b_i - a_i + 1 \big).$$
The longest side-length of $R$ is $\lg(R) := \max\{b_i - a_i + 1\}$, and the shortest side-length of $R$ is $\sht(R) := \min\{b_i - a_i + 1\}$.

A \emph{component} of a set $S \subset \ZZ^d$ is a maximal connected set in the graph $\ZZ^d[S]$ (the subgraph of $\ZZ^d$ induced by $S$). Given a subset $S \subset [n]^d \times [k]^\ell$, let $R(S)$ denote the smallest rectangle such that $S \subset R(S)$.

We next define the \emph{span} $\< A \>$ of a set $A$ in $C([n]^d \times [k]^\ell,2)$. The definition we give here is slightly different from that in~\cite{d=r=3}, but has many of the same properties (see Section~\ref{hiersec}). This definition simplifies the proof in Section~\ref{T1PfSec}.

\begin{defn}
Let $n,k \in \N$ and $A \subset C([n]^d \times [k]^\ell,2)$. Let $C_1, \ldots, C_m$ denote the collection of connected components in $[A]$. The \emph{span of $A$} is defined to be the following collection of rectangles:
$$\< A \> \; := \; \big\{ R(C_1), \ldots, R(C_m) \big\}.$$
If $[A]$ is connected (i.e., $m = 1$), then we say that $A$ \emph{spans} the rectangle $R(C_1)$. If $R \in \< A \cap R \>$, then we say $A$ \emph{internally spans} $R$.
\end{defn}

If $\< A \> = \{R\}$, i.e., $A$ spans $R$, then we shall write $\< A \> = R$. If $S$ is a set and $S \subset [A \cap S]$ then we shall say that $A$ \emph{internally fills} $S$. 

Given a set $S$, and $p \in [0,1]$, say that $A \sim \Bin(S,p)$ if the elements of $A \subset S$ are chosen independently at random with probability $p$. If $R$ is a rectangle in $C([n]^d \times [k]^\ell,r)$, then let
$$P_p(R) \; := \; \Pr_p\big(R \in \< A \cap R \> \big) \; = \; \Pr \big( R \in \< A \> \: | \:  A \sim \Bin(R,p) \big),$$
i.e., the probability that $A$ internally spans $R$.

A set is said to be \emph{occupied} if it is non-empty (i.e., contains some element of $A$), and it is said to be \emph{full} if every site is in $A$.
We shall use throughout the paper the notation
$$q \; := \; -\log(1 - p)$$
as in \cite{Hol}. Note that $p \sim q$ for small $p$. The advantage of this notation is the fact that
\begin{equation} \label{qbeta} \beta_k\big(1 - (1-p)^n\big) \; = \; e^{-g_k(nq)}.\end{equation}
Let $u(x) = 1 - e^{-qx}$ for any $x \in \RR$, and note that this is the probability that a set $S$ of size $x$ is empty (i.e., not occupied) under $\Pr_p$. Given $\x \in \RR^d$ and $j \in [d]$, we define 
$$u_j(\x) \; := \; u\Big( \prod_{i \neq j} x_i \Big),$$
and if $R \subset [n]^d \times [k]^\ell$ is a rectangle, then let $u_j(R) = u_j\big( \dim(R) \big)$.

We next recall the concept of \emph{disjoint occurrence} of events, and the van den Berg-Kesten Lemma~\cite{vBK}, which utilizes it. An event $E$ defined on subsets of $[N]$ is \emph{increasing} if $(S \subset T) \wedge E(S)$ implies $E(T)$. In the setting of bootstrap percolation on a graph $G$, two increasing events $E$ and $F$ \emph{occur disjointly} if there exist disjoint sets $S,T \subset V(G)$ such that the infected sites in $S$ imply that $E$ occurs, and the infected sites in $T$ imply that $F$ occurs. (We call $S$ and $T$ \emph{witness sets} for $E$ and $F$.) We write $E \circ F$ for the event that $E$ and $F$ occur disjointly.

\begin{vBKlemma}
Let $E$ and $F$ be any two increasing events defined in terms of the infected sites $A \subset V(G)$, and let $p \in (0,1)$. Then
$$\Pr_p(E \circ F) \; \le \; \Pr_p(E)\,\Pr_p(F).$$
\end{vBKlemma}

We remark here, for ease of reference, that there will be various constants which appear in the proof of Theorem~\ref{r=2}, which will depend on each other, but \emph{not} on $p$. These will be chosen in the order first $B$ (for `big'), then $\delta$, $k$, $Z$ (for `seed'), and finally $T$ (for `tiny'), and will satisfy
$$T \ll Z \ll \delta \ll 1 \ll B \ll k.$$
Each of these constants also depends on $d$, $\ell$ and $\eps$, which are fixed at the start of the proof.

\section{Sketch of the proof}\label{sketchsec}

To aid the reader's understanding, we shall give a brief outline of the proof of Theorem~\ref{sharp}; we begin with the base case of our induction on $r$, Theorem~\ref{r=2}. Let $G = C([n]^d \times [k]^\ell,2)$ and let $A \subset G$ be a random set, chosen with density
$$p \; = \; p(n) \; \le \; \left( \ds\frac{\lambda(d+\ell,\ell+2) - \eps}{\log n} \right)^{d-1}.$$
The first step is to apply a lemma introduced by Aizenman and Lebowitz~\cite{AL}, which says that if $\diam([A]) \ge B \log n$, then there exists an internally spanned rectangle $R$ in $G$ with $\frac{B \log n}{2} \le \lg (R) \le B \log n$. We shall show that $\Pr_p\big(R \in \< A \cap R \> \big)$, the probability that $R$ is internally spanned, is at most 
$$\exp\left( - \frac{d \lambda(d+\ell,\ell+2) - \eps}{p^{1/(d-1)}} \right) \; \le \; \frac{1}{n^{d+\eps}},$$
where the last inequality follows from our choice of $p$. This implies (by the union bound) that $\Pr_p\big( \diam([A]) \ge B \log n \big) = o(1)$.

To bound the probability that $R$ is internally spanned, we use the `hierarchy method', which was introduced by Holroyd~\cite{Hol}, and then adapted for our purposes in~\cite{d=r=3}. To be precise, we show that if $A \cap R$ spans $R$, then there is a `good and satisfied hierarchy' for $R$ (see Section~\ref{hiersec}, below), and so $\Pr_p\big(R \in \< A \cap R \> \big)$ is bounded by the expected number of such hierarchies. A hierarchy is essentially a way of breaking up the event $R \in \< A \cap R \>$ into a bounded number of \emph{disjoint} (and relatively simple) events, and so, by the van den Berg-Kesten inequality, the probability that a good hierarchy is satisfied is bounded by the product of the probabilities of these events (see Lemma~\ref{basic}). Moreover, we shall show that the number of good hierarchies is small (see Lemma~\ref{fewhiers}), and so it suffices to give a uniform bound on the probability that a good hierarchy is satisfied. 

To prove such a bound, the key step is to determine precisely the probability of `crossing a rectangle' $R'$ (see Section~\ref{crosssec}), that is, the probability that there is a path in $[A']$ across $R'$ in direction 1, where $A' = (A \cap R') \,\cup\, \{\x \,:\, x_1 \le a_1 - 1\}$. This is the most technical part of the paper, and we give a proof quite different from (and somewhat simpler than) that of the corresponding statement in~\cite{d=r=3}. One of the key steps is to partition $R'$ into pieces $S_j$ of bounded width, and study the probability of crossing $S_j$, under the coupling in which all elements of $R' \setminus S_j$ are already infected. In particular, our method allows us to avoid the use of Reimer's Theorem, which was a crucial tool in~\cite{d=r=3}. The required bound then follows (see the proof of Theorem~\ref{2r2tech} in Section~\ref{T2sec}) using some basic analysis, which generalizes results from~\cite{Hol} to higher dimensions (see Section~\ref{analsec}). 

Having proved Theorem~\ref{r=2}, we then deduce Theorem~\ref{sharp} using the method of Cerf and Cirillo~\cite{CC}, once again suitably generalized (see Section~\ref{CCsec}). Let $G = C([n]^d \times [k]^{\ell},r)$, and let $A \subset G$ be chosen randomly with density 
$$p \,=\, p(n) \,\le\, \left( \ds\frac{\lambda(d+\ell,\ell+r) - \eps}{\log_{(r-1)} n} \right)^{d-r+1}.$$ 
The first step is to observe that if $A$ internally spans $G$, then there exists a connected set $X$ with $X \subset [A \cap X]$ and $\log n \le \diam(X) = m \le 2\log n$ (see Lemma~\ref{smallcompt}). We consider the smallest cuboid $R$ containing $X$, and partition it into sub-cuboids $L_j$ of bounded width (along its longest edge). 

Now, we perform the bootstrap process in each $L_j$, under the coupling in which every vertex of $R \setminus L_j$ is already infected; under this coupling, the bootstrap structure on $L_j$ becomes isomorphic to $C([m]^{d-1} \times [k]^{\ell+1},r-1)$, and so we can apply the induction hypothesis. (In fact the situation is more complicated than this (see Theorem~\ref{genthm}), but we leave the details until Section~\ref{T1PfSec}.) By counting the expected number of minimal paths across $R$ (see Lemma~\ref{CClemma}), we deduce that the probability that $R$ is crossed by $[A \cap R]$ is at most $n^{-d-\eps}$, and hence with high probability there is no such connected set $X$ in $G$, in which case the set $A$ does not percolate, as required.

\section{Hierarchies}\label{hiersec}

In this section we shall recall (from~\cite{d=r=3} and~\cite{Hol}) the definition and some basic properties of a \emph{hierarchy} of a rectangle $R$. All of the results in this section were first proved by Holroyd~\cite{Hol} for $[n]^2$, and generalized to $C([n]^d \times [k]^\ell,r)$ in~\cite{d=r=3}. We refer the reader to those papers for detailed proofs, and note that although our definition of $\< A \>$ is slightly different from that in~\cite{d=r=3}, the proofs all work in exactly the same way.

We begin by defining a hierarchy of a rectangle in $C([n]^d \times [k]^\ell,2)$. If $G$ is an oriented graph, then let $N_G^\->(u) := \{v \in V(G) : u \to v\}$. 

\begin{defn}
Let $R$ be a rectangle in $C([n]^d \times [k]^\ell,2)$. A \emph{hierarchy} $\HH$ of $R$ is an oriented rooted tree $G_\HH$, with all edges oriented away from the root (`downwards'), together with a collection of rectangles $\{R_u : u \in V(G_\HH)\}$, $R_u \subset C([n]^d \times [k]^\ell,2)$, one for each vertex of $G_\HH$, satisfying the following criteria:
\begin{enumerate}
\item[$(a)$] The root of $G_\HH$ corresponds to $R$.
\item[$(b)$] Each vertex has at most $\ell + 2$ neighbours below it.
\item[$(c)$] If $u \to v$ in $G_\HH$ then $R_u \supset R_v$.
\item[$(d)$] If $N_{G_\HH}^\->(u) = \{v_1,\ldots,v_t\}$ and $t \ge 2$, then $\< R_{v_1} \cup \ldots \cup R_{v_t} \> = R_u$.
\end{enumerate}
\end{defn}

A vertex $u$ with $N_{G_\HH}^\->(u)=\emptyset$ is called a \emph{seed}. Given two rectangles $S \subset R$, we write $D(S,R)$ for the event (depending on the set $A \subset R$) that
$$R \: \in \: \< (A \cup S) \cap R \>,$$
i.e., the event that $R$ is internally spanned by $A \cup S$. Note that the event $D(S,R)$ depends only on the set $A \cap (R \setminus S)$, and let
$$P_p(S,R) \; := \; \Pr_p\big( D(S,R) \big).$$ 

We say a hierarchy \emph{occurs} (or is \emph{satisfied} by a set $A \subset R$) if the following events all occur \emph{disjointly}.
\begin{enumerate}
\item[$(e)$] If $u$ is a seed, then $R_u$ is internally spanned by $A$.
\item[$(f)$] If $(u,v)$ is such that $N_{G_\HH}^\->(u) = \{v\}$, then $D(R_v,R_u)$ holds.
\end{enumerate}
A hierarchy is \emph{good} for $(\hat{T},\hat{Z}) \in \RR^2$ if it satisfies the following.
\begin{enumerate}
\item[$(g)$] If $N_{G_\HH}^\->(u) = \{v\}$ and $|N_{G_\HH}^\->(v)| = 1$ then $\hat{T} \le \phi(R_u) - \phi(R_v) \le 2\hat{T}$.
\item[$(h)$] If $N_{G_\HH}^\->(u) = \{v\}$ and $|N_{G_\HH}^\->(v)| \neq 1$ then $\phi(R_u) - \phi(R_v) \le 2\hat{T}$.
\item[$(i)$] If $|N_{G_\HH}^\->(u)| \ge 2$ and $v \in N_{G_\HH}^\->(u)$, then $\phi(R_u) - \phi(R_v) \ge \hat{T}$.
\item[$(j)$] $u$ is a seed if, and only if, $\sht(R_u) \le \hat{Z}$.
\end{enumerate}
In our application we shall take $\hat{T} = T/p^{1/(d-1)}$ and $\hat{Z} = Z/p^{1/(d-1)}$ for some (small) constants $T,Z > 0$.  

The definition above is useful because of the following lemma, which says that if $A$ internally spans $R$, then there is a good hierarchy which is satisfied by $A$. Our definition of the span $\<A\>$ of the set $A$ is motivated by the proof of this lemma (see~\cite{d=r=3} for more details).

\begin{lemma}[Lemma~18 of~\cite{d=r=3}]\label{hierexists}
Let $A \subset C([n]^d \times [k]^\ell,2)$, let $\hat{T},\hat{Z} > 0$, and let $R \subset C([n]^d \times [k]^\ell,2)$ be a rectangle. Suppose that $A$ internally spans $R$. Then there exists a good (for $(\hat{T},\hat{Z})$) and satisfied hierarchy of $R$.
\end{lemma}

Given $\hat{T},\hat{Z} > 0$, let $\HH(R,\hat{T},\hat{Z})$ denote the collection of hierarchies for $R$ which are good for the pair $(\hat{T},\hat{Z})$. The next lemma makes the straightforward (but crucial) observation that there are only `few' possible hierarchies.

\begin{lemma}[Lemma~19 of~\cite{d=r=3}]\label{fewhiers}
Let $B,T,\hat{Z},p > 0$ and $n,d,k,\ell \in \N$. Let $R$ be a rectangle in $C([n]^d \times [k]^\ell,2)$ with $\textup{long}(R) \le B/p^{1/(d-1)}$, and let $\hat{T} = T/p^{1/(d-1)}$. Then there exists a constant $M = M(B,T,d,\ell)$ such that
$$|\HH(R,\hat{T},\hat{Z})| \; \le \; M\left( \frac{1}{p} \right)^M.$$
\end{lemma}

Finally, we state the following key lemma, which gives us our fundamental bound on the probability that $A$ percolates. The lemma follows easily from Lemma~\ref{hierexists} and the van den Berg-Kesten Lemma (see Lemma~20 of~\cite{d=r=3} or Section~10 of~\cite{Hol}). Recall that $P_p(R)$ denotes the probability that a rectangle $R$ is spanned by a set $A \sim \Bin(R,p)$.

\begin{lemma}[Lemma~20 of~\cite{d=r=3}]\label{basic}
Let $R$ be a rectangle in $C([n]^d \times [k]^\ell,2)$, and let $\hat{T},\hat{Z} > p > 0$. Then
$$\Pr_p\big(R \in \< A \cap R \> \big) \; \le \; \sum_{\HH \in \HH(R,\hat{T},\hat{Z})} \Bigg( \prod_{N_{G_\HH}^\->(u) = \{v\}} P_p(R_v,R_u) \Bigg) \Bigg( \prod_{\textup{ seeds }u} P_p(R_u) \Bigg).$$
\end{lemma}

\section{Crossing a rectangle}\label{crosssec}

In this section we shall bound from above the probability that a rectangle $R$ is `crossed' by a set $A \sim \Bin(R,p)$. Our bound (see Lemma~\ref{crossR}, below) is a generalization of Lemma~21 of~\cite{d=r=3}, but the proof will be somewhat simpler than that given in~\cite{d=r=3}; in particular, we shall avoid using Reimer's Theorem. We refer the reader also to the paper of Duminil-Copin and Holroyd~\cite{DCH}, where similar ideas are used.

We begin by fixing integers $d,\ell \in \N$, with $d \ge 2$. In order to save repetition, we shall keep these values fixed throughout the section. Let $G = C([n]^d \times [k]^\ell,2)$, where $n$ and $k$ will be chosen later.

A \emph{path in direction $j$} across a rectangle $R = [(a_1,\ldots,a_d),(b_1,\ldots,b_d)] \subset V(G)$ is a path in $G$ from a point in the set $\{\x \in R : x_j = a_j\}$ to a point in the set $\{\x\in R : x_j = b_j\}$. 

\begin{defn}
A rectangle $R = [(a_1,\ldots,a_d),(b_1,\ldots,b_d)]$ in $G = C([n]^d \times [k]^\ell,2)$ is said to be \emph{left-to-right crossed} in direction $j$ (or just crossed) by $A \subset V(G)$ if the set $A \cap R$ has the following property. Let
$$A' \; := \; (A \cap R) \,\cup\, \{\x \,:\, x_j \le a_j - 1\}.$$ Then there is path in $[A']$ across $R$ in direction $j$.
\end{defn}

We write $H^{\rightarrow(j)}(R)$ for this event, and define $H^{\leftarrow(j)}(R)$ (the event that $R$ is right-to-left crossed by $A$) similarly, with $x_j \le a_j - 1$ replaced by $x_j \ge b_j + 1$. As in~\cite{d=r=3}, we shall bound from above the function
$$h^{(j)}(R,t) \; := \; \max_{W \subset R, |W| \le t} \Big\{ \Pr_p \big( H^{\rightarrow(j)}(R) \: \big| \: W \subset A \big) \Big\}.$$
Note that $\Pr_p(H^{\rightarrow(j)}(R)) = h^{(j)}(R,0)$, and recall that $u(x) = 1 - e^{-qx}$. By symmetry, it suffices to bound $h^{(1)}(R,t)$.

\begin{lemma}\label{crossR}
Let $d,\ell \in \N$ and $B, \delta > 0$. If $k \in \N$ is sufficiently large then the following holds. Let $p > 0$ be sufficiently small, and let $R$ be a rectangle in $C([n]^d \times [k]^\ell,2)$, with $\dim(R) = (a_1, \ldots, a_d)$, where $a_i \le B/p^{1/(d-1)}$ for every $i \neq 1$.
Then, for any $t \in \N$,
$$\Big( \beta_{\ell+1}\big( u_1(R) \big) \Big)^{a_1+1} \; \le \; \Pr_p\big( H^{\rightarrow (1)}(R) \big) \; \le \; h^{(1)}(R,t) \; \le \; \Big( \beta_{\ell+1}\big(u_1(R) \big)  \Big)^{(1 - \delta)a_1 - k t},$$
where $u_1(R) = u\big( \prod_{i=2}^d a_i \big)$. 
\end{lemma}

As mentioned above, the strategy we shall use to prove this lemma differs from that in~\cite{d=r=3}. Instead of directly looking at the probability of this rectangle being left-to-right crossed, we will rather study the probability that a rectangle $S$ with dimensions $(s,a_2,\ldots, a_d)$, and with $r(\x)$ decreased by one for each $\x \in S$ with $x_1 = s$, is crossed from left to right in direction $1$, with $s$ large but \emph{constant} (so in particular $s \ll a_1$). Having proven an essentially sharp estimate for this rectangle, we shall be able to extend this bound to any length $a_1$, by splitting the large rectangle into rectangles of width $s$.

This point of view has the following advantage: it allows us to study the structure of the bootstrap process under the assumption that no two sites in $A \cap S$ are close to one another. In order to do so, we introduce the following slight generalization of the structure $C([n]^d \times [k]^{\ell},2)$. It corresponds to (or, more precisely, may be coupled with) the process inside the rectangle $S$ when everything in $R \setminus S$ is already infected.

\begin{figure}
\begin{center}
\includegraphics[width=0.60\hsize]{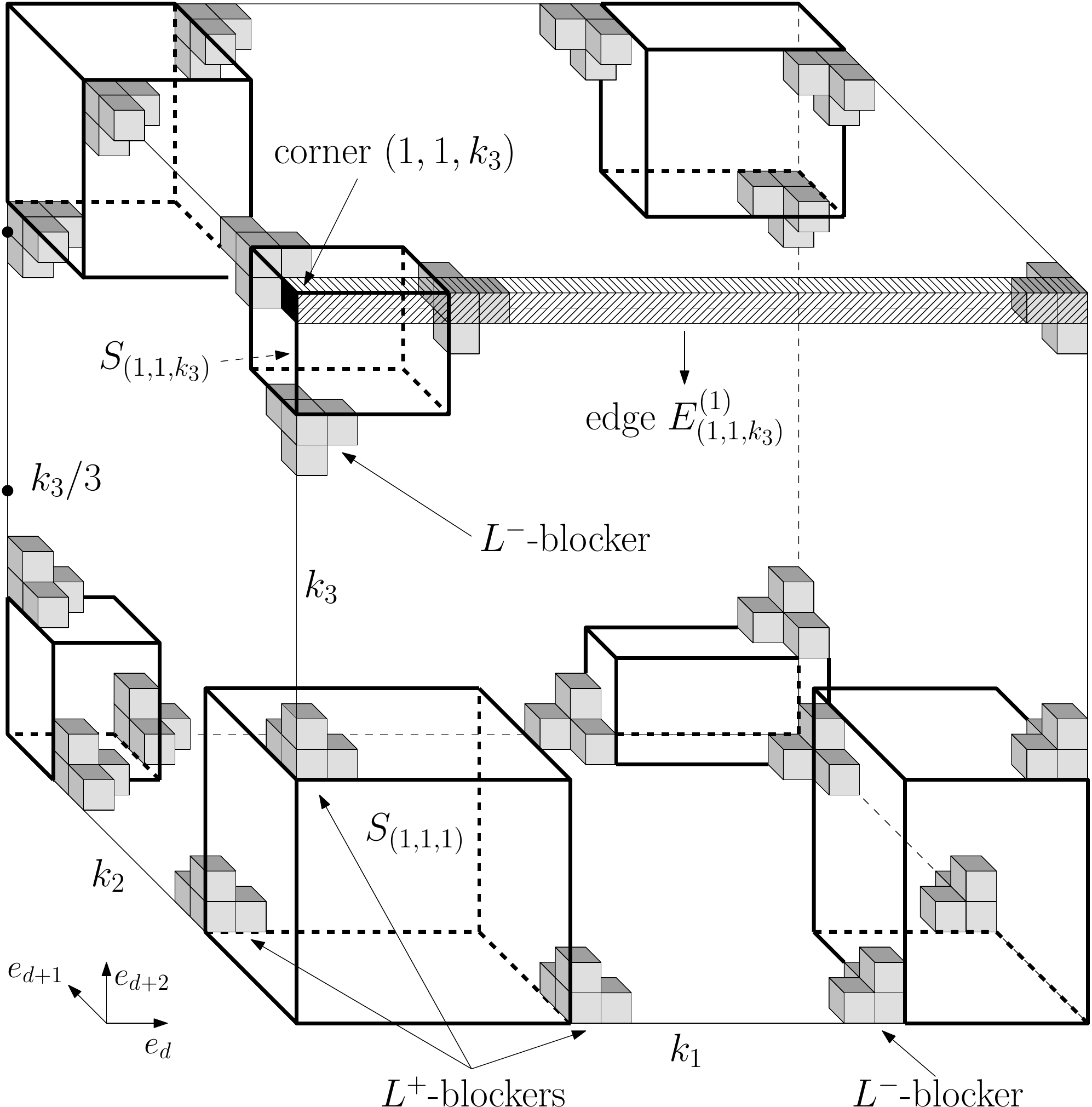}
\end{center}
\caption{The set $[m_1]\times\dots \times [m_{d-1}]\times [k_1]\times \dots \times [k_{\ell+1}]$, with examples of blockers and edges. Observe that the $d-1$ first dimensions are not depicted: each `unit' square is a set $\M_\x$.}
\label{fig:setS}
\end{figure}

Given vectors $\m \in \N^{d-1}$ and $\k \in \N^{\ell + 1}$, we define $C([\m] \times [\k],1)$ to be the bootstrap structure such that
\begin{enumerate}
\item[$(a)$] the vertex set is $S = [m_1] \times \dots \times [m_{d-1}] \times [k_1] \times \dots \times [k_{\ell +1}]$,\\[-2ex]
\item[$(b)$] the edge set is induced by $\ZZ^{d+\ell}$,\\[-2ex]
\item[$(c)$] $v = (a_1,\ldots,a_{d-1},b_1,\ldots,b_{\ell+1})$ has threshold $1 + \left| \big\{ j \in [\ell+1] : b_j \not\in \{1,k_j\} \big\} \right|$.
\end{enumerate}
We remark that in our applications, we shall take $k_2 = \dots = k_{\ell+1} = k$, and $k_1 = s$, where $k$ is much larger than $s$.

To study this structure, we slice the set $S = [\m] \times [\k]$ into sets $\M_\x$, see Figure~\ref{fig:setS}, where
$$\M_\x \; := \; \big\{ \y \in [\m] \times [\k] \,:\, y_{d-1+j} = x_j \text{ for every } j \in [\ell+1] \big\}$$
for each $\x \in [\k]$. Given a vector $\x \in \ZZ^d$, let $\C(\x) := \{1,x_1\} \times \dots \times \{1,x_d\}$ denote the set of `corners' of $\x$. Now, given a corner $\b \in \C(\k) = \{1,k_1\} \times \dots \times \{1,k_{\ell+1}\}$ of $S$, and a direction $j \in [\ell+1]$, we define a  \emph{boundary edge} (or simply an \emph{edge}) of $S$ to be the union of sets $\M_\x$ over those $\x$ with $x_i = b_i$ for $i \neq j$, so
$$E^{(j)}_\b \; = \; \bigcup_{t=1}^{k_j} \big\{ \M_\x \,:\, x_i = b_i \textup{ if } i \neq j \textup{ and } x_j = t \big\}.$$ 
Note that if $\b_i = \b'_i$ for each $i \neq j$, then $E^{(j)}_\b = E^{(j)}_{\b'}$.

We shall need the following generalization of the notion of \emph{blockers} from~\cite{d=r=3}. Let $e_i = (0,\ldots,0,1,0,\ldots,0)$ denote the vector with a single $1$ in position $i$.

\begin{defn}
Given $\b \in \C(\k)$ and $j \in [\ell+1]$, let $\x \in E^{(j)}_\b$. The set $\M_\x$ is an \emph{$L^+$-blocker} of the edge $E^{(j)}_\b$ if the events $U_\x$ and $\{V_\x^{(i)} : d \le i \le d+\ell\}$ all occur, where
\begin{align*}
U_\x  & \,:=\,  \big\{ \M_\x\textup{ is not occupied} \big\},\\[+1ex]
V_\x^{(i)} & \,:=\,  \left\{
\begin{array} {l}
\big\{ \M_{\x - e_i} \textup{ is not occupied} \big\}  \textup{ if } i \neq j \text{ and }b_i=k_i,\\[+1ex]
\big\{ \M_{\x + e_i} \textup{ is not occupied} \big\} \text{ otherwise}.
\end{array}\right.
\end{align*}
It is an \emph{$L^-$-blocker} of $E^{(j)}_\b$ if the event $V_\x^{(j)}$ in the definition above is replaced by the event
$$\hat{V}_\x^{(j)} \; = \; \{ \M_{\x - e_j} \textup{ is not occupied} \}.$$
The edge $E^{(i)}_\b$ is said to be \emph{blocked} if there exist $\y,\z \in E^{(i)}_\b$ such that $\M_{\y}$ is an $L^+$-blocker and $\M_{\z}$ is an $L^-$-blocker of $E^{(i)}_\b$, with $z_i > y_i$. It is said to be \emph{fully blocked} if moreover $y_i < (k_i/3) - 1$ and $z_i > (2k_i/3) + 1$. 
 \end{defn}
 
Note that $L$-blockers are so-named because of their shape; $L$ is not a variable. The following lemma is purely deterministic.

\begin{lemma}\label{detercross}
Let $\m \in \N^{d-1}$ and $\k \in \N^{\ell+1}$, let $A \subset C([\m] \times [\k],1)$, and let $S = [\m] \times [\k]$. Suppose that there is a path in $[A]$ across $S$ in direction $j$, for some $d \le j \le d+\ell$. Then one of the following holds:
\begin{itemize}
\item[$(a)$] $A$ contains two sites $x \neq y$ with $d_G(x,y) \le 2$.
\item[$(b)$] One of the boundary edges $\{ E^{(j)}_\b \,:\, \b \in \C(\k) \}$ is not blocked.
\item[$(c)$] One  of the boundary edges $\{ E^{(i)}_\b \,:\, j \neq i \in [\ell+1], \,\b \in \C(\k) \}$ is not fully blocked.
\end{itemize}
\end{lemma}

\begin{proof}
Suppose that none of the three events holds; that is, $A \subset S$ does not contain two sites at distance at most two from one another, and for every $\b \in \C(\k)$, the boundary edge $E^{(j)}_\b$ is blocked, and the boundary edges $\{ E^{(i)}_\b \,:\, j \neq i \in [\ell+1] \}$ are all fully blocked. 

We define a set $\hat{S}$ as follows (see Figure~\ref{fig:setS}). For each $\b \in \C(\k)$ and $i \in [\ell+1]$, let $y_i(\b)$ denote the minimal $i$-coordinate of an $L^+$-blocker $\M_\y$ of $E^{(i)}_\b$, and let $z_i(\b)$ denote the maximal $i$-coordinate of an $L^-$-blocker $\M_\z$ of $E^{(i)}_\b$. Note that $y_j(\b) < z_j(\b)$, and that $y_i(\b) < (k_i/3) - 1$ and $z_i(\b) > (2k_i/3) + 1$ for every $j \neq i \in [\ell+1]$.

For each $\b \in \C(\k)$, let 
$$S_\b \; := \; [\m] \times I_1(\b) \times \dots \times I_{\ell+1}(\b),$$
where $I_i(\b) = [1,y_i(\b) - 1]$ if $\b_i = 1$, and $I_i(\b) =  [z_{\ell+1}(\b) + 1,k_i]$ if $\b_i = k_i$. Let 
$$\hat{S} \; := \; \bigcup_\b S_\b.$$

\medskip
\noindent \ul{Claim}: $\hat{S} \cup A$ is a stable set, i.e., $[\hat{S} \cup A] = \hat{S} \cup A$.

\begin{proof}[Proof of claim]
We first claim that the sets $S_\b$ are pairwise at distance at least three. Indeed, let $\b,\b' \in \C(\k)$, and suppose that there exists some $\x \in [\k]$ such that $d(S_\b,\M_\x) \le 1$ and $d(S_{\b'},\M_\x) \le 1$. Suppose first that $\b_i = \b'_i$ for every $i \neq j$. Then either $\b = \b'$, or $\b_j = 1$ and $\b'_j = k_j$, say. But then $d(S_\b,\M_\x) \le 1$ implies that $\x_j \le y_j(\b)$, and $d(S_{\b'},\M_\x) \le 1$ implies that $\x_j \ge z_j(\b)$. Since $y_j(\b) < z_j(\b)$, this is a contradiction. 

So assume that $\b_i = 1$ and $\b'_i = k_i$ for some $i \neq j$. Since $d(S_\b,\M_\x) \le 1$, we have $\x_i \le y_i(\b)$, and since $d(S_{\b'},\M_\x) \le 1$, we have $\x_i \ge z_i(\b')$. But $y_i(\b) < (k_i/3) - 1 < (2k_i/3) + 1 < z_i(\b')$, which is a contradiction, and hence the sets $S_\b$ are pairwise at distance at least three, as claimed. 

Suppose that $[\hat{S} \cup A] \setminus \hat{S} \cup A$ is non-empty, and consider the first new site $v$ to be infected. It has at most one neighbour in $\hat{S}$, by the previous observation, and at most one neighbour in $A$, since $A$ does not contain two sites at distance at most two. Thus $v$ must have threshold at most two, and hence it belongs to an edge, $E^{(i)}_\b$, say.

Now, simply note that if a vertex of $E^{(i)}_\b$ is at distance one from $\hat{S}$, then it is in $\M_\y$ for some $\M_\y$ which is a blocker of $E^{(i)}_\b$. By the definition of a blocker, these vertices have no element of $A \setminus \hat{S}$ as a neighbour, and so it must have threshold one. But if $v$ has threshold one, then $v \in \M_\b$ for some $\b \in \C(\k)$, and since $v \not\in \hat{S}$ (by assumption), it follows that $S_\b$ is empty, and that $\M_\b$ is a blocker for $E^{(i)}_\b$, so $v$ has no neighbour in $\hat{S} \cup A$. This final contradiction proves the claim.
\end{proof}

It follows immediately from the claim that $[A] \subset \hat{S} \cup A$. But there is no path in $\hat{S} \cup A$ across $S$ in direction $j$, since the rectangles $S_\b$ are pairwise at distance at least three, and the elements of $A$ are pairwise at distance at least three. The lemma follows.
\end{proof}

We shall need the following bound on the probability that an edge is blocked.

\begin{lemma}\label{blocked}
Given $\m \in \N^{d-1}$ and $\k \in \N^{\ell+1}$, let $S = C([\m] \times [\k],1)$. Let $\b \in \C(\k)$, and let $j \in [\ell+1]$ and $p > 0$. Then
$$\Pr_p( E^{(j)}_\b \textup{ is not blocked} ) \; \le \; k_j \Big( \beta_{\ell+1}\big(u(S) \big) \Big)^{k_j - 2},$$
and
$$\Pr_p( E^{(j)}_\b \textup{ is not fully blocked} ) \; \le \; 2 \Big( \beta_{\ell+1}\big(u(S) \big) \Big)^{k_j/3 - 2},$$
where $u(S) = u\big( \prod_{i = 1}^{d-1} m_i \big)$.
\end{lemma}

In order to prove Lemma~\ref{blocked}, we shall need the following lemma from~\cite{d=r=3}, which is easily proved by induction on $m$. Given $\ell,m \in \N$, consider some sequence of events
$$\E \; = \; \Big\{U_i : i \in [m+1] \Big\} \cup \Big\{ V^{(i)}_j : i \in [\ell], j \in [m] \Big\}.$$
An $L$-gap in $\E$ is an event $\neg \left( U_i \vee U_{i+1} \vee V^{(1)}_i \vee \ldots \vee V^{(\ell)}_i \right)$ for some $i \in [m]$.

\begin{lemma}[Lemma~6 of~\cite{d=r=3}]\label{L6}
Let $\ell,m \in \N$, let $u \in (0,1)$, and suppose that each event in the set
$$\E \; = \; \Big\{U_i : i \in [m+1] \Big\} \cup \left\{ V^{(i)}_j : i \in [\ell], j \in [m] \right\}$$
occurs independently with probability $u$.

Let $L(m,u)$ denote the probability that there is no $L$-gap in $\E$. Then
$$\beta_{\ell+1}(u)^{m+1} \; \le \; L(m,u) \; \le \; \beta_{\ell+1}(u)^m,$$
where $\beta_{\ell+1}(u)$ is the function defined in the Introduction.
\end{lemma}

\begin{proof}[Proof of Lemma~\ref{blocked}]
Assume without loss that $\b_j = 1$, and let $\x(t) = \b + (t-1)e_j$ for each $t \in [k_j]$, so $E^{(j)}_\b = \{ \x(t) : t \in [k_j] \}$. For each $y \in [k_j]$, consider the following events:
\begin{align*}
F_1(y): & \; \textup{ $\M_{ \x(t) }$ is not an $L^+$-blocker of $E^{ (j) }_\b$ for each $1 \le t \le y-1$.}\\[+0.5ex]
F_2(y): & \; \textup{ $\M_{ \x(t) }$ is not an $L^-$-blocker of $E^{(j)}_\b$ for each $y+2 \le t \le k_j$.}
\end{align*}
Note that the events $F_1(y)$ and $F_2(y)$ are independent.

Suppose that $E^{(j)}_\b$ is not blocked. We claim there exists $y \in [k_j]$ such that $F_1(y)$ and $F_2(y)$ both hold; that is, there is no $L^+$-blocker $\M_{\x(t)}$ of $E^{(j)}_\b$ with $t < y$, and there is no $L^-$-blocker $\M_{\x(t)}$ of $E^{(j)}_\b$ with $t > y+1$. (To see this, simply take $y$ minimal such that $\M_{\x(y)}$ is an $L^+$-blocker of $E^{(j)}_\b$, or $k_j$ if there is no such blocker.)  

The lemma now follows from Lemma~\ref{L6}, applied to the events $U_{\x(t)}$ and $V_{\x(t)}^{(i)}$ for $t \in [k_j]$ and $j \neq i \in [\ell+1]$. Indeed, we have
$$\beta_{\ell+1}\big( u(S) \big)^{y+1} \; \le \; F_1(y) \; \le \; \beta_{\ell+1}\big( u(S) \big)^y,$$
and similarly for $F_2(y)$, and hence 
$$\Pr_p( E^{(j)}_\b \textup{ is not blocked} ) \; \le \; \sum_{y=1}^{k_j} \Pr_p\big(F_1(y) \wedge F_2(y) \big)  \; \le \; k_j \Big( \beta_{\ell+1}\big(u(S) \big) \Big)^{k_j -2},$$
as required. 

Finally, if $E^{(j)}_\b$ is not fully blocked then either the event $F_1\big( (k_j/3) - 1\big)$ or the event $F_2\big( (2k_j)/3 \big)$ occurs, and so
\begin{eqnarray*}
\Pr_p( E^{(j)}_\b \textup{ is not fully blocked} ) & \le & \Pr_p\Big( F_1\big( (k_j/3) - 1\big) \Big) \,+\, \Pr_p\Big(F_2\big( (2k_j)/3 \big) \Big) \\
& \le & 2 \Big( \beta_{\ell+1}\big(u(S) \big) \Big)^{k_j/3 - 2},
\end{eqnarray*}
as claimed.
\end{proof}

Given a rectangle $S \subset C([n]^{d-1} \times [k]^{\ell+1},1)$, a set $A \subset S$, and a direction $j \in [\ell+1]$, define the event
$$\X^S_j(A) \; := \; \big\{ \textup{there is a path in $[A \cap S]$ across $S$ in direction $d + j - 1$} \big\}.$$
The following upper bound on the probability of $\X^S_j(A)$ follows easily from Lemmas~\ref{detercross} and~\ref{blocked}. 

\begin{lemma}\label{probcross}
Let $B > 0$ and $d,k,\ell \in \N$, with $d \ge 2$. If $p > 0$ is sufficiently small then the following holds. Let $j \in [\ell+1]$, and let $\m \in \N^{d-1}$ and $\k \in \N^{\ell+1}$, with $m_i \le B/p^{1/(d-1)}$, $2\ell < k_j \le k/6$ and $k/2 \le k_i \le k$ for each $i \neq j$. Then
$$\Pr_p\big( \X^S_j(A) \big) \; \le \; 2^{\ell+1} k_j \Big( \beta_{\ell+1}\big(u(S) \big) \Big)^{k_j - 2},$$
where $S = C([\m] \times [\k],1)$ and $u(S) = u\big( \prod_{i = 1}^{d-1} m_i \big)$.
\end{lemma}

\begin{proof}
By Lemma~\ref{detercross}, if there is a path in $[A]$ across $S$ in direction $d + j - 1$, then either $A$ contains two sites within distance two, or one of the boundary edges in direction $j$ is not blocked, or one of the other boundary edges is not fully blocked. The probability that $A$ contains two sites at distance at most two is at most 
$$(2d + 2\ell)^2 |S| p^2 \; \le \; 4(d + \ell)^2 B^{d-1}k^{\ell+1} p \; = \; O(p).$$
There are $2^\ell$ boundary edges in direction $j$, so, by Lemma~\ref{blocked}, the probability that one of them is not blocked is at most
$$2^\ell k_j \Big( \beta_{\ell+1}\big(u(S) \big) \Big)^{k_j - 2}.$$
Finally, there are at most $2^\ell \ell$ boundary edges not in direction $j$, so the probability that one of them is not fully blocked is at most
$$2^{\ell+1} \ell \Big( \beta_{\ell+1}\big(u(S) \big) \Big)^{k/6 - 2},$$
by Lemma~\ref{blocked}, and since $k_i \ge k/2$ for every $i \neq j$. Since $p$ was chosen sufficiently small, and $k_j \le k/6$ and $k_j > 2\ell$, it follows that
$$\Pr\big( \X^S_j(A) \big) \; \le \; 2^{\ell+1} k_j \Big( \beta_{\ell+1}\big(u(S) \big) \Big)^{k_j - 2},$$
as required.
\end{proof}

We can now deduce Lemma~\ref{crossR} from Lemma~\ref{probcross}.

\begin{proof}[Proof of Lemma~\ref{crossR}]
The lower bound is straightforward, and follows by Lemma~7 of~\cite{d=r=3}, and the second inequality is immediate from the definition. We shall prove the upper bound. Let $R$ be a rectangle as described in the lemma, so $R \subset C([n]^d \times [k]^\ell,2)$ with $\dim(R) = (a_1, \ldots, a_d)$, where $a_i \le B/p^{1/(d-1)}$ for every $i \neq 1$. Recall that $B, \delta > 0$, and that $k \ge k_0(d,\ell,B,\delta)$ is sufficiently large. 

Let $A \sim \Bin(R,p)$, let $t \in \N$, and let $W \subset R$ with $|W| = t$. We are required to bound from above the probability that there is a path in $[A']$ across $R$ in direction $1$, where $A' \; = \; (A \cap R) \,\cup\, \{\x \,:\, x_1 \le a_1 - 1\} \cup W$.

Let $s = k/10$, $\m = (a_2,\ldots,a_d)$ and $\k = (s,k,\ldots,k) \in \N^{\ell+1}$, and assume for simplicity that $s$ divides  $a_1$. We partition $R$ into $M = a_1/s$ blocks $B_1, \ldots, B_M$, where $B_j \cong [\m] \times [\k]$ for each $j$, in the obvious way, i.e., so that $B_j$ is a translate of $[\m] \times [\k]$. Replace the thresholds in $B_j$ with those of $C([\m] \times [\k],1)$, and allow the bootstrap process to occur independently in each block. (By this, we mean that the blocks do not interact with each other.) Denote by $\{A\}(j)$ the closure of $A \cap B_j$ under this process, i.e., the closure of $A \cap B_j$ in the bootstrap structure $C([\m] \times [\k],1)$. 

Let $\{A\} = \bigcup_{j=1}^M \{A\}(j)$. The following claim shows that this is a coupling.\\[-1ex]

\noindent \ul{Claim}: $\{A\} \supset [A']$, where $[A']$ denotes the closure of $A'$ in $C([n]^d \times [k]^\ell,2)$.

\begin{proof}[Proof of claim]
Let $\x$ be a vertex of $B_j$, so $\x = (y_1,x_2,\ldots,x_d,y_2,\ldots,y_{\ell+1})$, where $x_j \in [m_j]$ and $y_j \in [k]$ for each $j \ge 2$, and $(j-1)s + 1 \le y_1 \le js$. Observe that $\x$ has at most one neighbour in $A' \setminus B_j$, since such a neighbour must differ from $\x$ in direction 1. Moreover, `internal' vertices of $B_j$ (those with $y_1 \not\in \{(j-1)s+1,js\}$) have no neighbours in $A'$ outside $B_j$. 

In the original system, $C([n]^d \times [k]^\ell,2)$, the threshold of vertex $\x$ was 
$$r(\x) \; = \; 2 \,+\, \left| \big\{ 2 \le j \le \ell + 1 : y_j \not\in \{1,k\} \big\} \right|.$$
In the coupled system, it is $r(\x) - 1 + I[y_1 \not\in \{(j-1)s+1,js\}]$. It follows that the threshold of no vertex has increased, and the threshold of those vertices which have a neighbour in $A'$ outside $B_j$ have decreased by one. Thus $\{A\} \supset [A']$, as claimed. 
\end{proof}

Let  $\J \subset [M]$ denote the set of indices $j$ such that $B_j \cap W \neq \emptyset$, and recall that $|\J| \le t$. Observe that, by the claim, if $R$ is left-to-right crossed by $A \cup W$ in direction 1, then the event $\X^{B_j}_1(A)$ occurs for each $j \in [M] \setminus \J$, i.e., there is a path in $\{A\} \cap B_j$ across $B_j$ in direction 1. Note, moreover, that the events $\X^{B_j}_1(A)$ for $j \in [M] \setminus \J$ are independent. Hence, by Lemma~\ref{probcross}, and recalling that $M = a_1/s$ and $|\J| \le t$,
\begin{align*}
& \Pr_p\Big( R\textup{ is left-to-right crossed by } A\text{ in direction }1 \: \big| \: W \subset A \Big) \; \le \; \prod_{j \in [M] \setminus \J} \Pr_p\big( \X^{B_j}_1(A) \big)\\
& \hspace{3.2cm} \; \le \;  \prod_{j \in [M] \setminus \J} 2^{\ell+1} s \Big( \beta_{\ell+1}\big(u_1(B_j) \big) \Big)^{s - 2} \; \le \; \Big( \beta_{\ell+1}\big(u_1(B_j) \big) \Big)^{(1 - \delta)a_1 - st},
\end{align*}
where $u_1(B_j) = u\big( \prod_{i = 2}^d a_i \big)$. In the final inequality we used the fact that $a_i \le B/p$ for each $i \neq 1$, so $\beta_{\ell+1}\big(u_1(B_j) \big)$ is bounded away from 1 (as a function of $B$, $d$ and $\ell$). Hence
$$\Big( 2^{\ell+1} s \Big)^{a_1/s} \Big( \beta_{\ell+1}\big(u_1(B_j) \big) \Big)^{- 2a_1/s} \; \le \; \Big( \beta_{\ell+1}\big(u_1(B_j) \big) \Big)^{-\delta a_1}$$
since $s = k/10$ is sufficiently large. This proves Lemma~\ref{crossR}.
\end{proof}

\section{Analytic tools}\label{analsec}

In this section we shall extend the analytic tools used by Holroyd~\cite{Hol} to the $d$-dimensional setting. We remark that the results of this section, together with the method of~\cite{Hol}, are sufficient to prove Theorem~\ref{sharp} in the case $r = 2$.  

The following line integral was introduced in~\cite{Hol} in the case $d = 2$. Let $\RR_+$ denote the (strictly) positive reals. Given any function $f : \RR_+ \to \RR_+$, and $\a,\b \in \RR^d_+$, define
$$W_f(\mathbf{a},\mathbf{b}) \; := \; \ds\inf_{\gamma \,:\, \mathbf{a} \to \mathbf{b}} \int_\gamma \bigg( \sum_j f\Big( \prod_{i \neq j} x_i \Big) \, \textup{d}x_j  \bigg),$$
where the infimum is taken over all piecewise linear, increasing paths from $\mathbf{a}$ to $\mathbf{b}$ in $\RR^d_+$ (see Section 6 of \cite{Hol}). Moreover, for any two rectangles $R \subset R' \subset [n]^d \times [k]^\ell$, let
\begin{equation}\label{Udef}
U_f(R,R') \; = \; W_f\Big( p^{1/(d-1)} \dim(R), p^{1/(d-1)} \dim(R') \Big).
\end{equation}
The aim of this section is to prove the following two propositions, which will allow us to deduce Theorem~\ref{r=2} from Lemmas~\ref{basic} and~\ref{crossR}. The first is a generalization of Lemma~37 of~\cite{Hol}.

\begin{prop}\label{pod}
Let $n,d,k,\ell \in \N$, let $\hat{T},\hat{Z},p > 0$, and let $Z = \hat{Z} \cdot p^{1/(d-1)}$. For any hierarchy $\HH$ of a rectangle $R \subset C([n]^d \times [k]^\ell,2)$ which is good for $(\hat{T},\hat{Z})$, there exists a rectangle $S = S(\HH) \subset R$, with
$$\phi(S) \; \le \; \sum_{\textup{seeds }u} \phi(R_u),$$
such that
$$\sum_{N_{G_\HH}^\->(u) = \{v\}} U_{g_{\ell+1}}(R_v,R_u) \; \ge \; U_{g_{\ell+1}}(S,R) \: - \: \Big(d\, p^{1/(d-1)} g_{\ell+1}(Z) \Big)\left| \left\{ u \in \HH : |N_{G_\HH}^\->(u)| \ge 2 \right\} \right|.$$
\end{prop}

The rectangle $S(\HH)$ is called the \emph{pod} of the hierarchy $\HH$. In order to understand this statement, ignore the final (error) term, and observe that the lemma gives us a lower bound on the sum a large number of small line integrals (which correspond to events $D(R_v,R_u)$ in the hierarchy). 

The next result, which is a generalization of Proposition~14 of~\cite{Hol}, shows that, if there are not too many big seeds, then this lower bound is exactly what we want. It will follow from the fact that the line integral $W_{g_\ell}(\a,\b)$ is minimized by following the main diagonal as closely as possible. 

Given a vector $\x \in \RR^d$, we shall write $\Delta(\x) = \max_j \{x_j\}$. Given two vectors $\a,\b \in \RR^d$, we shall write $\a \le \b$ if $a_j \le b_j$ for each $j \in [d]$, and $\a < \b$ if $a_j < b_j$ for every $j \in [d]$.

\begin{prop}\label{minW}
Let $d,\ell \in \N$, and let $\a,\b \in \RR^d_+$, with $\a \le \b$ and $\min_j \{b_j\} = b_i$. Then
$$W_{g_\ell}(\mathbf{a},\mathbf{b}) \; \ge \; d \int_{\Delta(\a)}^{\Delta(\b)} g_\ell(z^{d-1}) \,\dz \; - \; d \Delta(\b) g_\ell\Big(  \prod_{i \neq j} b_j \Big).$$
\end{prop}

\begin{rmk}
We shall use the following simple properties of the function $g_k(z)$ defined in the introduction: $g_k(u)$ is decreasing, convex and continuous, and $g_k(z) \le 2e^{-kz}$ if $z$ is sufficiently large. Note that either we have $\prod_{i \neq j} b_j \le \log\big( d \Delta(\b) \big)$, or $d \Delta(\b) g_k \big(  \prod_{i \neq j} b_j \big) \to 0$ as $\Delta(\b) \to \infty$.
\end{rmk}

We shall prove Propositions~\ref{pod} and~\ref{minW} using a discretization argument. Given a function $f$ and a path $\gamma$ in $\RR^d_+$, we shall write 
$$w_f(\gamma) \; := \; \int_\gamma \bigg( \sum_j f\Big( \prod_{i \neq j} x_i \Big) \, \textup{d} x_j  \bigg),$$
so that $W_f(\a,\b)= \ds\inf_{\gamma \,:\, \mathbf{a} \to \mathbf{b}} w_f(\gamma)$. We begin with a simple observation.

\begin{obs}\label{apppath}
Let $f : \RR_+ \to \RR_+$ be continuous, and let $\a,\b \in \RR^d_+$ with $\0 < \a \le \b$. For every $\eps > 0$, there exists a piecewise linear, increasing path $\gamma_\eps$ from $\a$ to $\b$, with each linear piece parallel to one of the axes, and all of equal length, such that 
$$w_f\big( \gamma_\eps \big) \; \le \; W_f(\a,\b) + \eps.$$
\end{obs}

The following lemma is a generalization of Lemma~18 of~\cite{Hol}.

\begin{lemma}\label{Wincr}
Let $f: \RR_+ \to \RR_+$ be continuous and decreasing, and let $\a,\b,\c \in \RR^d_+$ with $ \a \le \b \le \c$. Then
$$W_f(\a,\c) \; \ge \; W_f(\b,\c).$$
\end{lemma}

\begin{rmk}
Notice that a similarly `intuitive' inequality, that $W_f(\a,\b) \le W_f(\a,\c)$, is not true in general, even in two dimensions. To see this, consider for example the triple $\a = (1,1)$, $\b = (B,1)$ and $\c = (B,B)$, and let $B \gg 1$. Then $W_f(\a,\b) \to \infty$ as $B \to \infty$, but if $f$ is integrable then $W_f(\a,\c) = O(1)$. 
\end{rmk}

\begin{proof}
The proof will be by induction on $d$. When $d = 1$ it is trivial, since there is a unique path from $\a$ to $\c$, which passes through $\b$.

Let $d \ge 2$, and assume that the result holds for $d-1$. Let $\eps > 0$, and let $\gamma_\eps$ be the path from $\a$ to $\c$ given by Observation~\ref{apppath}. In other words, $\gamma_\eps$ is piecewise linear and increasing, with each linear piece parallel to one of the axes, and all of equal length, and $w_f(\gamma_\eps) \le W_f(\a,\c) + \eps$. 

Now consider the first point $\v$ on $\gamma_\eps$ such that $\v \ge \b$, and observe that $\v_j = \b_j$ for some $j \in [d]$. Assume that $j = 1$, let $\a' = (b_1,a_2,\ldots,a_d)$, and observe that $\a'$, $\b$ and $\v$ all live in the same $(d-1)$-dimensional hyperplane. Hence, by the induction hypothesis, it follows that $W_f(\a',\v) \ge W_f(\b,\v)$. 

Now, let $\gamma_1$ denote the section of $\gamma_\eps$ between $\a$ and $\v$, and let $\gamma_2$ denote the section from $\v$ to $\c$. Consider the path $\delta_1$ from $\a'$ to $\v$ obtained from $\gamma_1$ by projecting onto the hyperplane $x_1 = b_1$. Observe that each linear piece which is parallel to the $x_1$-axis disappears, and each other piece retains its length and direction, and has its $x_1$-coordinate increased. Since $f$ is decreasing, it follows that $w_f(\delta_1) \le w_f(\gamma_1)$. 

Now, since $W_f(\a',\v) \ge W_f(\b,\v)$, it follows that there exists a path $\delta'$ from $\b$ to $\v$ such that $w_f(\delta') \le w_f(\delta_1) + \eps$.

Finally, let $\delta_\eps$ denote the path from $\a'$ to $\c$ obtained by conjoining the paths $\delta'$ (from $\b$ to $\v$) and $\gamma_2$ (from $\v$ to $\c$). By the observations above, we have
$$w_f\big( \delta_\eps \big) \; \le \; w_f(\delta_1) + w_f(\gamma_2) + \eps \; \le \; w_f(\gamma_1) + w_f(\gamma_2) + \eps \; = \; w_f\big( \gamma_\eps \big) + \eps,$$
and hence
$$W_f(\b,\c) \; \le \; w_f\big( \delta_\eps \big) \; \le \; w_f\big( \gamma_\eps \big) + \eps \; \le \; W_f(\a,\c) + 2\eps,$$
by our choice of $\gamma_\eps$. Since $\eps > 0$ was arbitrary, the lemma follows. 
\end{proof}

We are now ready to prove Proposition~\ref{pod}. The proof is exactly as in~\cite{Hol}, except we need to replace Lemma~18 of~\cite{Hol} with Lemma~\ref{Wincr}, above. For completeness, we sketch the proof.

\begin{proof}[Proof of Proposition~\ref{pod}]
Let $f : \RR_+ \to \RR_+$ be continuous and decreasing; the lemma holds for any such function. The key step is a $d$-dimensional version of Proposition~15 of~\cite{Hol}, which states the following: for every $\a,\b,\c,\d \in \RR^d_+$ with $\a \le \b$ and $\c \le \d$, and every $x,Z \in \RR_+$ and $\r \in \RR^d_+$ with $\b,\d \le \r \le \b + \d + (x,\ldots,x)$, $x < Z$ and $\r \ge (2Z,\ldots,2Z)$, there exists $\s \in \RR^d_+$ with $\s \le \a + \c$ such that
$$W_f(\a,\b) + W_f(\c,\d) \; \le W_f(\s,\r) - (xd) f(Z).$$
This statement for $d = 2$ follows by Propositions~12 and~13 and Lemmas~17 and~18 of~\cite{Hol}. The first three generalize easily to the $d$-dimensional setting; in fact they are easy consequences of the fact that $f$ is decreasing. The last follows for general $d$ by Lemma~\ref{Wincr}. 

Proposition~\ref{pod} now follows by a straightforward induction argument, exactly as in the proof of Lemma~37 of~\cite{Hol}, noting that if $R$, $S$ and $T$ are rectangles with $R = \< S \cup T\>$ in $C([n]^d \times [k]^\ell,2)$, then $\dim(S) + \dim(T) \ge \dim(R) - (1,\ldots,1)$. 
\end{proof}

Finally, we prove Proposition~\ref{minW}. In this case the proof does not follow by the method of~\cite{Hol}, which was via an application of Green's Theorem in the plane. We shall discretize and apply Lemma~\ref{Wswitch}. Given two piecewise linear paths $\gamma$  and $\gamma'$ in $\RR^d_+$, we say that $\gamma'$ is a \emph{permutation} of $\gamma$ if it is obtained by permuting the linear pieces of $\gamma$. 

The following lemma allows us to permute adjacent linear pieces in order to move the path closer to the main diagonal.

\begin{lemma}\label{Wswitch}
Let $f : \RR_+ \to \RR_+$ be convex, let $ \a \in \RR^d_+$ and $b \in \RR_+$, and set $\b = \a + be_1$ and $\c = \a + be_2$. Suppose that $a_1 \le a_2$. Then
$$W_f(\a,\b) + W_f(\b,\b+\c) \; \le \; W_f(\a,\c) + W_f(\c,\b+\c).$$  
\end{lemma}

\begin{proof}
This follows easily from the definition. Since $f$ convex, we have
$$f(x) - f(x + z) \; \le \; f(y) - f(y + z)$$
for any $x,y,z \in \RR$ with $x \ge y$. Thus
\begin{align*}
& W_f(\a,\b) + W_f(\b,\b+\c) \; = \; b f\Big( a_2 \prod_{i \ge 3} a_i \Big) \, + \, b f\Big( (a_1 + b) \prod_{i \ge 3} a_i \Big)\\
& \hspace{2cm} \le \; b f\Big( (a_2 + b) \prod_{i \ge 3} a_i \Big) \, + \, b f\Big( a_1 \prod_{i \ge 3} a_i \Big) \; = \; W_f(\a,\c) + W_f(\c,\b+\c),
\end{align*}
by the inequality above with $x = \prod_{i \neq 1} a_i$, $y = \prod_{i \neq 2} a_i$, and $z = b \prod_{i \ge 3} a_i$. 
\end{proof}

\begin{proof}[Proof of Proposition~\ref{minW}]
Let $f: \RR_+ \to \RR_+$ be continuous and convex; the result will hold for any such function. Recall that $ \a,\b \in \RR^d_+$ with $\a \le \b$, and assume without loss of generality that $b_1 \le \ldots \le b_d$. We require a lower bound on $W_f(\a,\b)$. Let $B = b_d = \Delta(\b)$ and let $\b' = (B,\ldots,B)$. Observe that $\b' \ge \b$, so
$$W_f(\a,\b') \; \le \; W_f(\a,\b) \,+\, W_f(\b,\b').$$
It is easy to see that $W_f(\b,\b') \le (Bd) f\big( \prod_{j=2}^d b_j \big) = d \Delta(\b) f\big( \prod_{j=2}^d b_j \big)$ (simply choose a path which grows in direction 1 first, then direction 2, and so on), and so the proposition will follow from the statement
$$W_f(\a,\b') \; \ge \; d \int_{\Delta(\a)}^B f(z^{d-1}) \,\dz.$$

Let $\eps > 0$, and let $\gamma$ be a path from $\a$ to $\b' = (B,\ldots,B)$ given by Observation~\ref{apppath}. Thus $\gamma$ is piecewise linear and increasing, with each linear piece parallel to one of the axes, and all of equal length, and $w_f(\gamma) \le W_f(\a,\b') + \eps$. Let $\delta > 0$ denote the length of each piece of $\gamma$, and note that we may choose $\delta$ as small as we like. 

We claim that there exists a permutation $\gamma'$ of $\gamma$ which passes within $\ell_\infty$-distance $\delta$ of every point of the straight line between $(A,\ldots,A)$ and $(B,\ldots,B)$, such that  
$$w_f(\gamma) \; \ge \; w_f(\gamma') \; \ge \; d \int_{\Delta(\a)}^B f(z^{d-1}) \,\dz \,-\, \eps.$$
This follows by Lemma~\ref{Wswitch}. Indeed, let $\gamma'$ be chosen to minimize $w_f(\gamma')$ over all permutations of $\gamma$. Assume, without loss of generality, that $a_1 \le a_2 \le \dots \le a_d$, and consider the piecewise linear path $\zeta$, given by 
$$(a_1,\ldots,a_d) \to \dots \to (a_j,\ldots,a_j,a_{j+1},\ldots,a_d) \to \dots \to (a_d,\ldots,a_d) \to (B,\ldots,B),$$
where $\x \to \y$ means that $\zeta$ follows a straight line between $\x$ and $\y$. By Lemma~\ref{Wswitch}, we can choose $\gamma'$ to be the permutation which follows $\zeta$ as closely as possible. The second inequality follows because $f$ is continuous, and we chose $\delta > 0$ sufficiently small.

Putting the pieces together, we have
$$W_f(\a,\b') \; \ge \; w_f(\gamma) - \eps \; \ge \; w_f(\gamma') - \eps \; \ge \; d \int_{\Delta(\a)}^B f(z^{d-1}) \,\dz \,-\, 2\eps.$$
Since $\eps > 0$ was arbitrary, the result follows.
\end{proof}

To finish the section, we prove the following simple property of $\lambda(d,r)$, which will be useful in Section~\ref{T2sec}.

\begin{prop}\label{<d/2}
Let $d,\ell \in \N$ with $d \ge 2$. Then $\lambda(d+\ell,\ell+2) < \ds\frac{d+1}{2}$.
\end{prop}

\begin{proof}
Recall from~\eqref{lamdef} that 
$$\lambda(d+\ell,\ell+2) \; = \; \int_0^\infty g_{\ell+1}(z^{d-1}) \, \dz.$$
By Proposition~3 of~\cite{d=r=3}, we have $g_{k+1}(z) \le g_k(z)$ for every $k \in \N$ and $z \in (0,\infty)$, so it suffices to prove the result for $\ell = 0$. By Theorem~5 of~\cite{HLR}, we have
$$\int_{\log(3/2)}^\infty g_1(z) \, \dz \; = \; \int_0^{2/3} \frac{-\log\big( \beta_1(1 - x) \big)}{x} \, \textup{d} x \; < \; \frac{\pi^2}{36} \; < \; \frac{1}{3}.$$
Moreover, since $g_1$ is decreasing, we have $g_1(z^{d-1}) \le g_1(z)$ whenever $z \ge 1$. Hence
$$\int_1^\infty g_1(z^{d-1}) \, \dz \; \le \; \int_1^\infty g_1(z) \, \dz \; \le \; \int_{\log(3/2)}^\infty g_1(z) \, \dz \; < \; \frac{1}{3}.$$

To bound the integral when $z < 1$, observe that $\beta_1(u) \ge \sqrt{u}$ for $0 \le u \le 1$, and that $1 - e^{-z} \ge (1 - 1/e)z$ for $0 \le z \le 1$, so
$$\beta_1(1 - e^{-z}) \; \ge \; \sqrt{1 - e^{-z}} \; \ge \; \left( 1 - \frac{1}{e} \right)^{1/2} \sqrt{z}$$
for every $0 \le z \le 1$. Hence $g_1(z^{d-1}) \le -\log\left( z^{(d-1)/2} \right) + \ds\frac{1 - \log(e-1)}{2}$, and so
$$\int_0^1 g_1(z^{d-1}) \, \dz \; \le \; \frac{d-1}{2} \int_0^1 -\log z  \, \dz \,+\, \frac{2}{5}  \; = \; \frac{d-1}{2}\,+\, \frac{2}{5}.$$
Thus $\lambda(d+\ell,\ell+2) \le \ds\frac{d-1}{2} + \frac{2}{5} + \frac{1}{3} < \frac{d+1}{2}$, as claimed.
\end{proof}

\section{Proof of Theorem~\ref{r=2}}\label{T2sec}

In this section we complete the proof of Theorem~\ref{r=2}. We shall follow the basic method of Holroyd~\cite{Hol} (see also Sections~4.3 and~4.4 of~\cite{d=r=3}), but we shall need some new ideas here also. Theorem~\ref{r=2} will follow easily from the following theorem (see Corollary~\ref{2r2cor}).

\begin{thm}\label{2r2tech}
For every $d,\ell \in \N$ with $d \ge 2$, and every $\eps > 0$, there exists $B_0 > 0$ and $k_0: \N \to \N$ such that the following holds for every $B \ge B_0$ and every $k \ge k_0(B)$. 

Let $G = C([n]^d \times [k]^\ell,2)$, and let $p > 0$ be sufficiently small. Let $R \subset V(G)$ be a rectangle with $\lg(R) = B/p^{1/(d-1)}$. Then
$$\Pr_p\big( R \in \< A \cap R \> \big) \; \le \; \exp\left( - \frac{d \lambda(d+\ell,\ell+2) - \eps}{p^{1/(d-1)}} \right).$$
\end{thm}

We begin by bounding the probability that a rectangle grows sideways by $T/p^{1/(d-1)}$. Let $R \subset R'$ be rectangles in $C([n]^d \times [k]^\ell,2)$, and recall from Section~\ref{hiersec} that $P_p(R,R') = \Pr_p\big( D(R,R') \big)$, where $D(R,R')$ denotes the event that $R'$ is internally spanned by $(A \cup R) \cap R'$. 

We shall deduce the following lemma from Lemma~\ref{crossR}. We refer the reader to~\cite{GH2} (see Lemma~5) where a similar trick is used.

\begin{lemma}\label{DRR'}
For each $d,\ell \in \N$ and $B,\delta > 0$, there exist constants $k \in \N$, $Z = Z(k) > 0$ and $T  = T(k,Z) > 0$ such that the following holds.

Let $p > 0$ be sufficiently small, and let $R \subset R'$ be rectangles in $C([n]^d \times [k]^\ell,2)$ with $\dim(R) = (m_1,\ldots,m_d)$ and $\dim(R') = (m_1+s_1,\ldots,m_d+s_d)$. Suppose that $Z/p^{1/(d-1)} \le m_j \le B/p^{1/(d-1)}$ and $s_j \le T/p^{1/(d-1)}$ for each $j \in [d]$. Then
\begin{eqnarray}
P_p(R,R') & \le & \exp\bigg( - \big(1 - 2\delta \big) \sum_{j=1}^d g_{\ell+1}\Big( q \prod_{i \neq j} m_i \Big) s_j \bigg).\label{DRR}
\end{eqnarray}
\end{lemma}

\begin{proof}
For each direction $j \in [d]$, let $R^{<}_j$ denote the rectangle $\{\x \in R' : x_j < y_j \textup{ for all } \y \in R\}$, and similarly let $R^{>}_j$ denote the rectangle $\{\x \in R' : x_j > y_j \textup{ for all } \y \in R\}$. Write $R_j = R^{<}_j \cup R^{>}_j$, and let $C = \bigcup_{i < j} R_i \cap R_j$ denote the corner areas of $R' \setminus R$. Finally, let $W = A \cap C$, and let $t = |W|$. 

If the event $D(R,R')$ occurs, then clearly the events $H^{\-> (j)} (R^{>}_j)$ and $H^{\leftarrow (j)} (R^{<}_j)$ must also occur for each $j \in [d]$. Hence, 
\begin{eqnarray*}
P_p(R,R') & \le & \Pr_p\left( \bigwedge_j H^{\-> (j)} (R^{>}_j) \wedge H^{\leftarrow (j)} (R^{<}_j) \right).
\end{eqnarray*}
Note that $|C| \le {d \choose 2} T^2 B^{d-2} p^{-d/(d-1)}$. The idea is that, since $T$ may be chosen small compared with $Z$ (and also $B$, $d$, $k$, $\ell$), it is likely that $|A \cap C|$ will be small compared with $s_j$, and so the events $H^{\-> (i)}(R^{>}_i)$ and $H^{\leftarrow (j)}(R^{<}_j)$ are `almost independent'.

To be precise, by Lemma~\ref{crossR}, and the binomial theorem, we have 
\begin{eqnarray*}
P_p(R,R') & \le & \sum_{t=0}^{|C|} {{|C|} \choose t} p^t (1-p)^{|C|-t} \prod_j h^{(j)}(R^{>}_j,t) \cdot h^{(j)}(R^{<}_j,t)\\
& \le & \sum_{t=0}^{|C|} {{|C|} \choose t} p^t \prod_j \big( \beta_{\ell+1}\big( u_j(R) \big) \big)^{(1 - \delta) s_j - kt}\\
& = & \left( \prod_j \Big( \beta_{\ell+1}\big( u_j(R) \big)  \Big)^{(1-\delta)s_j} \right)  \left( 1 + p \prod_j \Big( \beta_{\ell+1}\big( u_j(R) \big) \Big)^{-k} \right)^{|C|}.
\end{eqnarray*}

To estimate the error term, we use our bounds on $m_j$ and $s_j$. Indeed, since $m_j \ge Z/p^{1/(d-1)}$ for each $j \in [d]$, we have $u_j(R) \ge 1 - \exp( - Z^{d-1} ) \ge Z^{d-1}/2$. Since $\beta_{\ell+1}(u)$ is increasing on $[0,1]$ (see Proposition~3 of~\cite{d=r=3}), and $\beta_{\ell+1}(u) \ge u$ when $u$ is small, it follows that
$$\prod_j \Big( \beta_{\ell+1}\big( u_j(R) \big) \Big)^{-k} \; \le \; \prod_j \big( 2 / Z^{d-1} \big)^k \; \le \; \left( \frac{1}{Z} \right)^{kd^2}.$$
Let $T_1 = p^{1/(d-1)}  \max_j \{ s_j \} \le T$, and recall that $|C| \le {d \choose 2} T_1^2 B^{d-2} p^{-d/(d-1)}$, since $m_j \le B / p^{1/(d-1)}$ for each $j \in [d]$. Hence, since $T_1 \le T$,
\begin{align*}
& \left( 1 + p \prod_j \Big( \beta_{\ell+1}\big( u_j(R) \big) \Big)^{-k} \right)^{|C|} \; \le \; \exp\Big( |C| Z^{-kd^2} p \Big) \; \le \; \exp\left( \frac{T_1^{3/2}}{p^{1/(d-1)}} \right)\\
& \hspace{3cm} \le \; \exp\bigg( \delta g_{\ell+1}\big( 2B^{d-1} \big) \max_j \{s_j\}  \bigg) \; \le \; \exp\bigg( \delta \sum_{j=1}^d g_{\ell+1}\Big( q \prod_{i \neq j} m_i  \Big) s_j  \bigg),
\end{align*} 
if $T > 0$ is chosen to be sufficiently small (with respect to $d$, $\ell$, $B$, $\delta$, $k$ and $Z$). The penultimate inequality follows since we may we choose $\sqrt{T} \le \delta g_{\ell+1}\big( 2B^{d-1} \big)$. In the final inequality, we used the facts that $g_{\ell+1}$ is decreasing, and that $q \prod_{i \neq j} m_i  \le 2B^{d-1}$. 

Finally, recall that $e^{-g_{\ell+1}(qx)} = \beta_{\ell+1}\big( u(x) \big)$, so 
$$\prod_j \Big( \beta_{\ell+1}\big( u_j(R) \big)  \Big)^{(1-\delta)s_j}  \; = \; \exp\bigg( - \big(1 - \delta \big) \sum_{j=1}^d g_{\ell+1}\Big( q \prod_{i \neq j} m_i  \Big) s_j \bigg).$$
Combining these bounds, the lemma follows.
\end{proof}

We now rewrite the right-hand side of \eqref{DRR} in a more useful form. We shall use the following easy observation from~\cite{Hol}.

\begin{obs}[Proposition 12 of \cite{Hol}]\label{Uanaltriv}
If $f$ is decreasing, then
$$W_f(\mathbf{a},\mathbf{b}) \; \le \; \sum_{j=1}^d \big( b_j - a_j \big) f\Big( \prod_{i \neq j} a_i \Big).$$
\end{obs}

By Observation~\ref{Uanaltriv} and the definition of $U_{g_{\ell+1}}(R,R')$, we have
$$\frac{1}{p^{1/(d-1)}}  U_{g_{\ell+1}}(R,R') \; \le \; \sum_{j=1}^d g_{\ell+1} \Big( q \prod_{i \neq j} m_i  \Big) s_j.$$
The following corollary of Lemma~\ref{DRR'} is now immediate.

\begin{cor}\label{DRRcor}
Under the conditions of Lemma~\ref{DRR'},
$$P_p(R,R') \; \le \; \exp \left( - \big( 1 - 2\delta \big) \frac{U_{g_{\ell+1}}(R,R')}{p^{1/(d-1)}}  \right).$$
\end{cor}

Next we bound the probability that a seed is internally spanned. Recall that $\phi(R)$ denotes the semi-perimeter of a rectangle $R$. 

\begin{lemma}\label{seeds}
Let $d,\ell,k \in \N$, let $B, \alpha > 0$, and let $Z  = Z(d,\ell,B,\alpha,k) > 0$ be sufficiently small. Let $p > 0$, and let $R$ be a rectangle in $C([n]^d \times [k]^\ell,2)$. Suppose $\sht(R) \le Z/p^{1/(d-1)}$ and $\lg(R) \le B/p^{1/(d-1)}$. Then
$$\Pr_p(R \in \< A \cap R \>) \; \le \; e^{- \alpha \phi(R)}.$$
\end{lemma}

\begin{proof}
Let $\dim(R) = (u_1,\ldots,u_d)$, and suppose without loss of generality that $u_1 = \lg(R)$ and $u_2 = \sht(R)$. Note that if $R \in \< A \cap R \>$, then $R$ has no `double gap', i.e., no pair of adjacent empty hyperplanes (see Lemma~27 of~\cite{d=r=3}). Thus, 
$$\Pr\big( R \in \< A \cap R \> \big) \; \le \; \Big(  2 k^\ell u_1^{d-2} u_2 p \Big)^{u_1/2} \; \le \; \Big( 2k^\ell B^{d-2} Z \Big)^{\phi(R)/2d} \; \le \; e^{-\alpha \phi(R)}$$
if $2k^\ell B^{d-2} Z \le e^{-2d \alpha}$, which holds if $Z > 0$ is sufficiently small, as required.
\end{proof}

Finally, we recall the following lemma from~\cite{d=r=3} (see also~\cite{AL}).

\begin{lemma}[Lemma~16 of~\cite{d=r=3}]\label{k2k}
Let $A \subset C([n]^d \times [k]^\ell,2)$. If $1 \le L \le \diam([A])$, then there exists a rectangle $R$, internally spanned by $A$, with
$$L \; \le \; \textup{long} (R) \; \le \; 2L.$$
\end{lemma}

We are ready to prove Theorem~\ref{2r2tech}. 

\begin{proof}[Proof of Theorem~\ref{2r2tech}]
Let $d,\ell \in \N$, with $d \ge 2$, and let $\eps > 0$. We choose positive constants $B$, $\alpha$, $\delta$, $k$, $Z$ and $T$ (chosen in that order), with $B > 0$ sufficiently large, $\delta > 0$ sufficiently small, and $k$, $Z$ and $T$ chosen so that Lemmas~\ref{DRR'} and~\ref{seeds} hold. In particular, let $\alpha = d \lambda(d+\ell,\ell+2)B$, and note that 
$$T \ll Z \ll \delta \ll 1 \ll B \ll k.$$
Finally, we let $p \to 0$, so that $p \ll T$. Let $\hat{T} = T/p^{1/(d-1)}$ and $\hat{Z} = Z/p^{1/(d-1)}$.

Let $R$ be a rectangle in $G = C([n]^d \times [k]^\ell,2)$ with $\dim(R) = (b_1,\ldots,b_d)$, let $\lg(R) = B/p^{1/(d-1)}$, and assume without loss of generality that $b_1 \le \ldots \le b_d$. By Corollary~\ref{DRRcor} and Lemmas~\ref{basic} and~\ref{seeds}, we obtain
\begin{align}
& \Pr_p\big(R \in \< A \cap R \>\big) \; \le \; \sum_{\HH \in \HH(R,\hat{T},\hat{Z})} \left( \prod_{N_{G_\HH}^\->(u) = \{v\}} P_p(R_v,R_u) \right) \left( \prod_{\textup{ seeds $u \in \HH$}} P_p(R_u) \right) \nonumber\\
& \hspace{1cm} \le \; \sum_{\HH \in \HH(R,\hat{T},\hat{Z})} \exp\left( -\sum_{N_{G_\HH}^\->(u) = \{v\}} \frac{(1 - 2\delta)U_{g_{\ell+1}}(R_v,R_u)}{p^{1/(d-1)}} - \alpha \sum_{\textup{seeds $u \in \HH$}} \phi(R_u) \right) \label{T2eq1}
\end{align}
For each hierarchy $\HH \in \HH(R,\hat{T},\hat{Z})$, let
$$Q(\HH) \; := \; \exp\left( -\sum_{N_{G_\HH}^\->(u) = \{v\}} \frac{(1 - 2\delta)U_{g_{\ell+1}}(R_v,R_u)}{p^{1/(d-1)}} - \alpha \sum_{\textup{seeds $u \in \HH$}} \phi(R_u) \right).$$
The theorem will follow easily from~\eqref{T2eq1}, Lemma~\ref{fewhiers} and the following claim.

\medskip
\noindent \textbf{Claim}: $Q(\HH) \le  \exp\left( - \ds\frac{d\lambda(d+\ell,\ell+2) - \eps}{p^{1/(d-1)}} \right)$ for every $\HH \in \HH(R,\hat{T},\hat{Z})$.

\begin{proof}[Proof of claim]
We shall consider three cases. First, suppose that $\HH$ has `many' seeds.

\medskip
\noindent \textbf{Case 1}: $\ds\sum_{\textup{ seeds }u} \phi(R_u) \ge \ds\frac{1}{Bp^{1/(d-1)}}$.

\medskip
In this case it is sufficient to consider only the second term in $Q(\HH)$. Indeed, 
$$\exp\left( - \alpha \sum_{\textup{ seeds }u} \phi(R_u) \right) \; \le \; \exp\left(- \frac{\alpha}{Bp^{1/(d-1)}} \right) \; = \; \exp\left( -\ds\frac{d\lambda(d+\ell,\ell+2)}{p^{1/(d-1)}} \right)$$
if $p > 0$ is sufficiently small, since $\alpha = d\lambda(d+\ell,\ell+2)B$, as required.

\medskip
Next, suppose that $R$ is unusually `long and thin'. Let $\grr(R) = p \prod_{j=2}^{d} b_j$, and recall that $b_d = B/p^{1/(d-1)}$, and that $B$ is chosen to be sufficiently large.

\medskip
\noindent \textbf{Case 2}: $\ds\sum_{\textup{ seeds }u} \phi(R_u) \le \ds\frac{1}{Bp^{1/(d-1)}}$ and $\grr(R) \le 2\log(Bd)$.

\medskip
In this case we consider only the first term in $Q(\HH)$. Let $S = S(\HH)$ be the pod of $\HH$, given by Proposition~\ref{pod}. Note that $\HH$ has bounded height (in terms of $B$, $d$ and $T$), and hence that $|G_\HH|$ is bounded (as $p \to 0$). Hence, by Proposition~\ref{pod}, we have
\begin{equation}\label{T2eq2}
Q(\HH) \; \le \; \exp\left( - \frac{ (1 - 2\delta) U_{g_{\ell+1}}(S,R)}{p^{1/(d-1)}}  + M_1 \right), 
\end{equation}
for some constant $M_1 = M_1(d,\ell,B,Z,T)$. 

Next, note that $b_d/b_1 \ge \sqrt{B}$, since $\grr(R) \le 2\log(Bd) \le \sqrt{B}$ and so 
$$b_1 \; \le \; \left( \frac{\grr(R)}{p} \right)^{1/(d-1)} \; \le \; \left( \frac{\sqrt{B}}{p} \right)^{1/(d-1)} \; \le \; \frac{\sqrt{B}}{p^{1/(d-1)}},$$ 
whereas $b_d = B/p^{1/(d-1)}$. Recall~\eqref{Udef}, the definition of $U_{g_{\ell+1}}(S,R)$, and recall also that $\phi(S) \le \sum_{\textup{seeds}} \phi(R_u)$, and that $g_{\ell+1}$ is decreasing. We obtain
$$U_{g_{\ell+1}}(S,R) \; \ge \; \Big( B - \phi(S) p^{1/(d-1)} \Big) g_{\ell+1}\bigg( q \prod_{j = 1}^{d-1} b_j \bigg) \; \ge \; \big( B - 1 \big) g_{\ell+1}\left( \frac{4\log(Bd)}{\sqrt{B}} \right) \; \ge \; 2B,$$
if $B > 0$ is sufficiently large. The first inequality above follows by considering growth only in direction $d$. For the second step, note that $q \prod_{j = 1}^{d-1} b_j \le  2 \cdot \grr(R) (b_1/b_d)$, and use the upper bounds on $\sum_{\textup{seeds}} \phi(R_u)$, $\grr(R)$ and $b_1/b_d$. The final inequality holds if $B$ is sufficiently large, since $g_{\ell+1}(z)\to \infty$ as $z \to 0$.

Thus, combining this bound with~\eqref{T2eq2}, we deduce that
$$Q(\HH) \; \le \; \exp\left( - \frac{B}{p^{1/(d-1)}}  + M_1 \right) \; < \; \exp\left( -\frac{d\lambda(d+\ell,\ell+2)}{p^{1/(d-1)}} \right)$$
if $B > 0$ is sufficiently large and $p > 0$ is sufficiently small, as required.

\medskip
Finally, we arrive at the main case.

\medskip
\noindent \textbf{Case 3}: $\ds\sum_{\textup{ seeds }u} \phi(R_u) \le \ds\frac{1}{Bp^{1/(d-1)}}$ and $\grr(R) \ge 2\log(Bd)$.

\medskip
The inequality~\eqref{T2eq2} again follows from Proposition~\ref{pod} and the definition of $Q(\HH)$, exactly as in Case~2. By Proposition~\ref{minW} (applied to the vectors $\a = p^{1/(d-1)} \dim(S)$ and $\b = p^{1/(d-1)} \dim(R)$), we have
$$U_{g_{\ell+1}}(S,R) \; \ge \; d \int_{1/B}^B g_{\ell+1} \big( z^{d-1} \big) \, \dz \: - \: (Bd) g_{\ell+1}\Big( p \prod_{j = 2}^d b_j \Big),$$
and hence, recalling that $\grr(R) \ge 2\log(Bd)$ and noting that $\delta / p^{1/(d-1)} \gg M_1$,
$$Q(\HH) \; \le \; \exp\left[ - \frac{(1 - 3\delta)}{p^{1/(d-1)}} \left( d \int_{1/B}^B g_{\ell+1} \big( z^{d-1} \big) \, \dz \: - \: (Bd) g_{\ell+1}\Big( 2\log(Bd) \Big) \right) \right].$$
But $g_{\ell+1}(z) \le 2e^{-(\ell+1)z}$ for large $z$ (see~\cite{d=r=3}), so
\begin{eqnarray*}
Q(\HH) & \le & \exp\left[ - \frac{(1 - 3\delta)}{p^{1/(d-1)}} \left( d \int_{1/B}^B g_{\ell+1} \big( z^{d-1} \big) \, \dz \: - \: \frac{2}{Bd} \right) \right]\\
& \le & \exp\left( -\frac{d\lambda(d+\ell,\ell+2) - \eps}{p^{1/(d-1)}} \right)
\end{eqnarray*}
if $B$ is sufficiently large, and $\delta$ is sufficiently small, as required.
\end{proof}

To complete the proof of Theorem~\ref{2r2tech}, recall that, by Lemma~\ref{fewhiers},
$$|\HH(R,\hat{T},\hat{Z})| \; \le \; M_2\left( \frac{1}{p} \right)^{M_2}.$$
for some constant $M_2 = M_2(B,T,d,\ell)$. Hence, by~\eqref{T2eq1} and the claim,
\begin{align*}
& \Pr_p\big(R \in \< A \cap R \>\big) \; \le \; \sum_{\HH \in \HH(R,\hat{T},\hat{Z})} Q(\HH) \; \le \; M_2\left( \frac{1}{p} \right)^{M_2} \max_{\HH \in \HH(R,\hat{T},\hat{Z})} Q(\HH)\\
& \hspace{1cm} \le \;  M_2\left( \frac{1}{p} \right)^{M_2}  \exp\left( - \ds\frac{d\lambda(d+\ell,\ell+2) - \eps}{p^{1/(d-1)}} \right) \; \le \;  \exp\left( - \ds\frac{d\lambda(d+\ell,\ell+2) - 2\eps}{p^{1/(d-1)}} \right)
\end{align*}
if $p > 0$ is sufficiently small. Since $\eps > 0$ was arbitrary, the theorem follows.
\end{proof}

We complete this section by deducing the following corollary of Theorem~\ref{2r2tech}, which is the technical statement which we shall need in Section~\ref{T1PfSec}. 

\begin{cor}\label{2r2cor}
Let $d,\ell \in \N$, with $d \ge 2$. If $\eps > 0$ is sufficiently small, then there exist 
$B > 0$ and $k_0 = k_0(B) > 0$ such that, if $k \ge k_0$ and $n \ge n_0(B,d,k,\ell,\eps)$ is sufficiently large, then the following holds. Let $G = C([n]^d \times [k]^\ell,2)$, and let
$$p \,=\, p(n) \,\le\, \left( \ds\frac{\lambda(d+\ell,\ell+2) - \eps}{\log n} \right)^{d-1}.$$ 
Then
$$\Pr_p\Big( \diam([A]) \ge B \log n \Big) \; \le \; n^{-\eps}.$$
\end{cor}

\begin{proof}
Let $\eps' = \eps'(d,\ell,\eps) > 0$ be sufficiently small, let $B_0 = B_0(\eps')$, $k_0 = k_0(B_0,\eps')$ be chosen according to Theorem~\ref{2r2tech}, and write $\lambda = \lambda(d+\ell,\ell+2)$. Let $n \in \N$ and, noting that the probability is monotone in $p$, let 
$$p \, = \, p(n) \, = \, \left( \ds\frac{\lambda - \eps}{\log n} \right)^{d-1}.$$ 
Let $B_1 = 2B_0 / (\lambda - \eps)$. We shall show that
$$\Pr_p\Big( \lg (R) \ge B_1\log n \textup{ for some } R \in \< A \> \Big) \; \le \; n^{-\eps}.$$
The result will then follow, since $\diam([A]) \le \max \big\{ \lg(R) : R \in \< A \> \big\}$.

Suppose $\lg(R) \ge B_1 \log n$ for some $R \in \< A \>$. By Lemma~\ref{k2k}, there exists an internally spanned rectangle $R' \subset R$ with 
$$\frac{B_1 \log n}{2} \; \le \; \lg (R') \; \le \; B_1 \log n.$$ 
By our choice of $p$, it follows that $\lg(R') = B/p^{1/(d-1)}$ for some $B \in [B_0,2B_0]$. Hence, by Theorem~\ref{2r2tech}, if $n$ is sufficiently large then
$$\Pr_p\big( R' \in \< A \cap R' \> \big) \; \le \; \exp\left( - \frac{d\lambda - \eps'}{p^{1/(d-1)}} \right) \; \le \; \exp\left( - \left( \frac{d\lambda - \eps'}{\lambda - \eps} \right) \log n \right) \; \le \; n^{-(d+\eps+\eps')}.$$
The last inequality holds since $\eps > 0$ and $\eps'  = \eps'(d,\ell,\eps) > 0$ were chosen sufficiently small, and because $\lambda(d+\ell,\ell+2) < (d+1)/2 < d$, by Proposition~\ref{<d/2}.

There are at most $(B\log n)^d n^d \le n^{d+ \eps'}$ potential such rectangles $R'$. So, writing $Y(B_1)$ for the number of internally spanned rectangles $R' \subset C([n]^d \times [k]^\ell,2)$ with $(B_1/2) \log n \le \lg (R') \le B_1 \log n$, we get
\begin{eqnarray*}
\Pr_p\big( \lg (R) \ge B_1\log n \textup{ for some } R \in \< A \> \big) \; \le \; \Ex_p \big( Y(B_1) \big) \; \le \; n^{-\eps}
\end{eqnarray*}
as required.
\end{proof}

It is easy to see that Corollary~\ref{2r2cor} implies Theorem~\ref{r=2}.

\section{The Cerf-Cirillo Method}\label{CCsec}

In this section we shall recall a fundamental technique in the study of bootstrap percolation on $[n]^d$. This technique was introduced by Cerf and Cirillo~\cite{CC}, and later used and refined by Cerf and Manzo~\cite{CM}, Holroyd~\cite{Holddim}, and Balogh, Bollob\'as and Morris~\cite{d=r=3}. We shall use this `Cerf-Cirillo method' in order to prove the induction step in our proof of Theorem~\ref{sharp}.

In order to state the main lemma of this section, we need to recall some definitions from~\cite{d=r=3}. We will be interested in two-coloured graphs, i.e., simple graphs with two types of edges, which we shall label `good' and `bad'. We call such a two-coloured graph `admissible' if it either contains at least one bad edge, or if every component is a clique (i.e., a complete graph). For any set $S$, let
$$\Lambda(S) \; := \; \big\{ \textup{admissible two-coloured graphs with vertex set } S \times [2]\big\}.$$
Now, given $m \in \N$, let
$$\Omega(S,m) \; := \; \big\{\P = (G_1, \ldots, G_m) \,:\, G_t \in \Lambda(S)\textup{ for each }t \in [m]\big\},$$
the set of sequences of two-coloured admissible graphs on $S \times [2]$ of length $m$.
We shall sometimes think of $G_t$ as a coloured graph on $S \times [2t-1,2t]$, and trust that this will cause no confusion. We shall be interested in probability distributions on $\Omega(S,m)$ in which, with high probability, there are bad edges in only very few of the graphs $G_t$.

Now, for each $\P \in \Omega(S,m)$, let $G_\P$ denote the graph with vertex set $S \times [2m]$, and the following edge set $E(G_\P)$ (see, for example, Figure~2).
\begin{eqnarray*}
& (a) & G_\P[S \times \{2y-1,2y\}] = G_y,\hspace{9cm}\\[+1ex]
& (b) & \{(x,2y),(x',2y+1)\} \in E(G_\P) \Leftrightarrow x = x',\\[+1ex]
& (c) & \{(x,y),(x',y')\} \notin E(G_\P)\textup{ if }|y - y'| \ge 2.
\end{eqnarray*}

Edges in $G_\P$ of type $(a)$ are labelled good and bad in the obvious way, to match the label of the corresponding edge in $G_y$. Thus $G_\P$ has three types of edge: good, bad, and unlabelled.

Such a graph $G_\P$, with $S = [3]$ and $m = 4$, is pictured below. Note that, for example, $G_3$ has two edges: $\{(1,1),(2,1)\}$ and $\{(3,1),(3,2)\}$, and that $G_4$ must contain a bad edge.

\vspace{0.2in}
\[ \unit = 1.2cm \hskip -9\unit
\varpt{3000} \pt{0}{0} \pt{0}{-1} \pt{0}{-2}
\point{1.5}{0} {$ \pt{0}{0} \pt{0}{-1} \pt{0}{-2} $}
\point{2.5}{0} {$ \pt{0}{0} \pt{0}{-1} \pt{0}{-2} $}
\point{4}{0} {$ \pt{0}{0} \pt{0}{-1} \pt{0}{-2} $}
\point{5}{0} {$ \pt{0}{0} \pt{0}{-1} \pt{0}{-2} $}
\point{6.5}{0} {$ \pt{0}{0} \pt{0}{-1} \pt{0}{-2} $}
\point{7.5}{0} {$ \pt{0}{0} \pt{0}{-1} \pt{0}{-2} $}
\point{9}{0} {$ \pt{0}{0} \pt{0}{-1} \pt{0}{-2} $}
\medline \dl{1.5}{0}{2.5}{0} \dl{1.5}{-1}{2.5}{-1} \dl{1.5}{-2}{2.5}{-2}
\point{2.5}{0} {$ \dl{1.5}{0}{2.5}{0} \dl{1.5}{-1}{2.5}{-1} \dl{1.5}{-2}{2.5}{-2} $}
\point{5}{0} {$ \dl{1.5}{0}{2.5}{0} \dl{1.5}{-1}{2.5}{-1} \dl{1.5}{-2}{2.5}{-2} $}
\dl{0}{0}{1.5}{0} \dl{0}{-1}{1.5}{0} \dl{0}{-2}{1.5}{0} \dl{0}{-1}{0}{-2}
\dl{0}{0}{0}{-1} \dl{1.5}{-1}{1.5}{-2} \dl{2.5}{-2}{4}{-2} \dl{2.5}{0}{4}{-1} \dl{2.5}{-1}{4}{0}
\dl{5}{0}{5}{-1} \dl{5}{-2}{6.5}{-2} 
\dl{7.5}{-2}{9}{-2}
\bez{0}{0}{-0.5}{-1}{0}{-2} \dl{7.5}{0}{9}{-1} \dl{9}{0}{9}{-1}
\point{0.5}{0.5}{\small $G_1$} \point{3}{0.5}{\small $G_2$}
\point{5.5}{0.5}{\small $G_3$} \point{8}{0.5}{\small $G_4$}
\point{0.9}{-3}{Figure 2: A graph $G_\P$, with $S = [3]$ and $m = 4$.}
\]

Given $G \in \Lambda(S)$, let $E^g(G)$ denote the set of good edges, and $E^b(G)$ denote the bad edges, so that $E(G) = E^g(G) \cup E^b(G)$. If $uv$ is a good edge in $G$, then we shall write $u \sim v$. For each vertex $v = (x,y) \in V(G_\P)$, let
$$\Gamma_\P(v) \; := \; \{ u \in V(G_\P) \,:\, u \sim v \textup{ and } u \neq v \},$$
and let $d_\P(v) = |\Gamma_\P(v)|$. We emphasize that $d_\P(v)$ is the number of \emph{good} edges incident with $v$.

Finally, let $X(\P)$ denote the event that there is a connected path across $G_\P$ (i.e., a path from the set $S \times \{1\}$ to the set $S \times \{2m\}$). Observe that the event $X(\P)$ holds for the graph $G_\P$ depicted in Figure~2.

The following lemma was first stated in~\cite{d=r=3}, but the proof is due to Cerf and Cirillo~\cite{CC}. 

\begin{lemma}[Cerf and Cirillo~\cite{CC}, see Lemma~35 of~\cite{d=r=3}]\label{CClemma}
For each $0 < \alpha < 1/2$ and $\eps > 0$, there exists $\delta > 0$ such that the following holds for all $m \in \N$ and all finite sets $S$ with $\alpha^4 |S|^\eps \ge 1$.

Let $\P = (G_1,\ldots,G_m)$ be a random sequence of admissible two-coloured graphs on $S \times [2]$, chosen according to some probability distribution $f_\Omega$ on $\Omega(S,m)$. Suppose $f_\Omega$ satisfies the following conditions:
\begin{itemize}
\item[$(a)$] Independence: $G_i$ and $G_j$ are independent if $i \neq j$,\\[-2ex]
\item[$(b)$] BK condition: For each $t \in [m]$, $r \in \N$, and each $x_1,y_1,\ldots,x_r,y_r \in V(G_t)$,
$$\Pr\left( \bigwedge_{j=1}^r \big( x_j \sim y_j \big) \wedge \bigwedge_{j \neq j'} \big( x_j \not\sim x_{j'} \big) \wedge \big( E^b(G_t) = \emptyset \big) \right) \; \le \; \prod_{j=1}^r \Pr\big( x_j \sim y_j \big),$$
\end{itemize}
and for each $t \in [m]$ and $v \in V(G_\P)$,
\begin{itemize}
\item[$(c)$] Bad edge condition: \,$\Pr\big( E^b(G_t) \neq \emptyset \big) \, \le \, |S|^{-\eps}$,\\[-2ex]
\item[$(d)$] Good edge condition: \,$\Ex\left( d_\P(v) \right) \, \le \, \delta$.\\[-2ex]
\end{itemize}
Then
$$\Pr\big(X(\P)\big) \; \le \; \alpha^m |S|.$$
\end{lemma}

\begin{rmk}
We shall apply Lemma~\ref{CClemma} with $S = [N]^{d-1} \times [k]^\ell$, where $N \le \log n$. The pair $uv$ will be an edge of the graph $G_t$ if $u,v$ are in the same component of $[A]$, where the closure is in the structure $C([N]^{d-1} \times [k]^{\ell+1},r-1)$, and $A$ is chosen according to $\Pr_p$. Edges will be labelled `good' if both endpoints lie in some internally filled component of `small' diameter, i.e., less than $B \log N$, for some suitably chosen $B > 0$. Condition $(b)$ will be proved using the van den Berg-Kesten Lemma, and conditions $(c)$ and $(d)$ by the induction hypothesis. The base cases are Corollary~\ref{2r2cor} and Lemma~\ref{goodedge}, below.
\end{rmk}

Given a bootstrap structure $G$ on $[n]^d \times [k]^\ell$, a set $A \subset V(G)$, a vertex $x \in V(G)$ and a number $m > 0$, we define the set
\begin{align}
& \Gamma_G(A,m,x) \: := \: \big\{ y \in V(G) \,:\, \textup{there exists an internally filled connected} \nonumber \\
& \hspace{3.25cm} \textup{component } X \subset V(G) \textup{ such that } x,y \in X \textup{ and }\diam(X) \le m \big\}.\label{Gamdef}
\end{align}
(This definition is important, and is due to Holroyd~\cite{Holddim}.) The following straightforward lemma, which we shall use to bound the expected number of good edges incident with a vertex, was proved in~\cite{d=r=3}.

\begin{lemma}[Lemma~36 of~\cite{d=r=3}]\label{goodedge}
Let $n,d,k,\ell \in \N$, with $d \ge 2$, and let $B > 0$. There exists a constant $c(B,d,k,\ell)$ such that the following holds. Given $p > 0$ sufficiently small, let $G = C([n]^d \times [k]^\ell,2)$ and  $A \sim \Bin(V(G),p)$. Then
$$\Ex_p\Big( \big| \Gamma_G(A, B/p^{1/(d-1)}, v) \big| \Big) \; \le \; c(B,d,k,\ell)\,\big(\log(1/p) \big)^{3d + \ell + 1}\,p$$
for every $v \in V(G)$.
\end{lemma}

We shall also use the following easy lemma from~\cite{CC}.

\begin{lemma}\label{smallcompt}
Let $A \subset C([n]^d \times [k]^{\ell},r)$. Then for every $1 \le L \le \diam([A])$, there exists a connected set $X$ which is internally filled, i.e., $X \subset [A \cap X]$, with $$L \; \le \; \diam(X) \; \le \; 2L.$$
\end{lemma}

\begin{proof}
Add newly infected sites one by one, and note that in each step the largest diameter of a component in $[A]$ may jump from at most $L - 1$ to at most $2L - 1$. Thus, at some point in the process the required set $X$ must appear as a component.
\end{proof}

\section{Proof of Theorem~\ref{sharp}}\label{T1PfSec}

We can now prove the following generalization of the lower bound in Theorem~\ref{sharp} by induction on $r$, using the method of Cerf and Cirillo for the induction step, and with Corollary~\ref{2r2cor} and Lemma~\ref{goodedge} as the base case. 

Recall from Section~\ref{CCsec} the definition~\eqref{Gamdef} of $\Gamma_G(A,m,x)$, the set of vertices which are connected to $x$ by a `small' component which is internally filled by $A$. We shall show that, for appropriate values of $p$ and $m$, the expected size of this set goes to zero as $n \to \infty$.

\begin{thm}\label{genthm}
Let $d,\ell,r \in \N$ with $d \ge r \ge 2$. If $\eps > 0$ is sufficiently small, then there exist $B > 0$ and $k_0 = k_0(B) > 0$ such that, if $k \ge k_0$ and $n \in \N$ is sufficiently large, then the following holds. Let $G = C([n]^d \times [k]^\ell,r)$, and let 
$$p \,=\, p(n) \,\le\, \left( \ds\frac{\lambda(d+\ell,\ell+r) - \eps}{\log_{(r-1)} n} \right)^{d-r+1}.$$ 
Then
$$\Pr_p\Big( \diam([A]) \ge B\log n \Big) \, \le \, n^{-(r-2)d - \eps},$$
and moreover 
$$\Ex_p\Big( \big| \Gamma_G(A, B\log n, v) \big| \Big) \, = \, o(1)$$
as $n \to \infty$, for every $v \in V(G)$.
\end{thm}

\begin{proof}
The proof is by induction on $r$; we begin by proving the base case, $r = 2$. Let $B = B^{(2)}(d,\ell,\eps)$ and $k_0(B)$ be given by Corollary~\ref{2r2cor}. The first statement follows from Corollary~\ref{2r2cor}, and the second follows by Lemma~\ref{goodedge}, so in this case we are done.

Let $r \ge 3$, and assume that the theorem holds for $r - 1$, for all $d,\ell \in \N$ and every sufficiently small $\eps > 0$. We shall prove the theorem with $B = B^{(r)}(d,\ell,\eps) = 1$ when $r \ge 3$. Fix $d,\ell \in \N$ and $\eps > 0$, let $p = p(n) > 0$ be as described above, and let $k \ge k_0(d,\ell,r,\eps) \in \N$ be sufficiently large. 

Let $G = C([n]^d \times [k]^\ell,r)$, and recall that $P_p(R)$ denotes the probability that a rectangle $R \subset V(G)$ is internally spanned by $A \sim \Bin\big( V(G),p\big)$. The induction step is a straightforward consequence of the following claim.

\medskip
\noindent \textbf{Claim}: If $R \subset C([n]^d \times [k]^\ell,r)$ is a rectangle and $\diam(R) = m \le \log n$, then 
$$P_p(R) \le \alpha^{m/k} (m + k)^{d+\ell}$$
for some $\alpha = \alpha(n) \to 0$ as $n \to \infty$.

\begin{proof}[Proof of claim]
If $m \le \log_{(r)}n$ then $P_p(R) \le k^{d+\ell} p \to 0$ as $n \to \infty$, since $R$ must contain an element of $A$. So assume that $m > \log_{(r)}n \gg k$, let $R' \supset R$ be a rectangle in $G$, with $R' \cong [m]^d \times [k]^\ell$, and let $t = \lfloor m/k \rfloor$. Assume without loss of generality that $\dim(R)_1 = m$ (i.e., $R$ has length $m$ in direction 1), and assume for simplicity that $m$ is divisible by $k$. We partition the rectangle $R'$ into blocks $L_1, \ldots, L_t$, each of size $[m]^{d-1} \times [k]^{\ell+1}$. To be precise, let $L_j = \big\{\x \in R' : (j-1)k + 1 \le x_1 \le jk \big\}$. 

Since $R$ is internally spanned by $A$, there exists a path in $[A \cap R]$ from the set $\{\x \in R : x_1 = 1\}$ to the set $\{\x \in R : x_1 = m\}$. We shall use Lemma~\ref{CClemma} to show that this is rather unlikely. In order to do so we use the following coupling. 

Replace the thresholds in each block $L_j$ with those of $C([m]^{d-1} \times [k]^{\ell+1},r-1)$, and run the bootstrap process independently in each block. Denote by $\{A\}(j)$ the closure of $A \cap L_j$ under this process, i.e., the closure in the bootstrap structure $C([m]^{d-1} \times [k]^{\ell+1},r-1)$.

Let $\{A\} = \bigcup_j \{A\}(j)$. The following subclaim shows that this is indeed a coupling. 

\medskip
\noindent \textbf{Subclaim}: $\{A\} \supset [A \cap R'] \supset [A \cap R]$.

\begin{proof}[Proof of subclaim]
Note that each vertex of $L_j$ has at most one neighbour in $R' \setminus L_j$, and `internal' vertices of $L_j$ (those with $x \not\in \{(j-1)k+1,jk\}$) have no neighbours outside $L_j$. A vertex $\x \in L_j$ originally had threshold $r$, and now (in the coupled system) has threshold $r - 1 + I[x_1 \not\in \{(j-1)k+1,jk\}]$. Thus, the threshold of no vertex has increased, and the threshold of those vertices which have a neighbour in $R'$ outside $L_j$ have decreased by one. Thus $\{A\} \supset [A \cap R']$, as claimed. The second inclusion follows since $R \subset R'$.
\end{proof}

Now, let $S = [m]^{d-1} \times [k]^\ell$, and for each $j \in [t]$, define a two-coloured graph $G_j$ on $S \times [2]$ as follows. For each $\x \in S \times [2]$, let $\tilde{\x}$ denote the element of $\{(j-1)k+1,jk\} \times [m]^{d-1} \times [k]^\ell$ corresponding to $\x$ in the natural isomorphism. Now define the edges of $G_j$ by
$$\x\y \in E(G_j) \textup{ if and only if $\tilde{\x}$ and $\tilde{\y}$ are in the same component of }\{A\}(j),$$
and let
\begin{align*}
& \x \sim \y \; \Leftrightarrow \textup{ there exists an internally filled connected component } \\
& \hspace{3cm} X \subset \{A\}(j) \textup{ such that } \x,\y \in X \textup{ and }\diam(X) \le B\log n,
\end{align*}
where $\x \sim \y$ means $\x\y$ is a `good' edge, as in Section~\ref{CCsec}, and $B = B^{(r-1)}(d-1,\ell+1,\eps)$ was chosen above. Note that $G_j$ is admissible, since $\x \sim \y$ and $\y \sim \z$ in $G_j$ implies that $\x$ and $\z$ are in the same component of $\{A\}(j)$, and so either $\x \sim \z$, or $\x\z$ is a bad edge. Note also that the event $\x \sim \y$ is increasing. 

For each set $A \subset V(G)$, we have defined a sequence $\P := (G_1,\ldots,G_t) \in \Omega(S,m)$ of admissible two-coloured graphs. We claim that the (random) sequence $\P$ satisfies the conditions of Lemma~\ref{CClemma}. Indeed, recall that $m \le \log n$, so
$$p \; \le \; \left( \ds\frac{\lambda(d+\ell,\ell+r) - \eps}{\log_{(r-2)} m} \right),$$
and let $\eps' = \eps / (d+\ell)$. By the induction hypothesis (and our choice of $B$ and $k$), for each $j \in [t]$ we have
\begin{equation}\label{condc} 
\Pr\big( E^b(G_j) \neq \emptyset \big) \; \le \; \Pr_p\Big( \textup{diam}\big( \{A\}(j) \big) > B \log m \Big) \; \le \; m^{-\eps} \; \le \; |S|^{-\eps'},
\end{equation}
since $|S| = m^{d-1} k^\ell \le m^{d+\ell}$. 

Next, choose a function $\alpha = \alpha(n)$ such that $\alpha \to 0$ sufficiently slowly as $n \to \infty$, and let $\delta = \delta(\alpha,\eps') > 0$ be given by Lemma~\ref{CClemma}. Since $\alpha(n) \to 0$ sufficiently slowly, and $d$, $\ell$ and $\eps$ are constants, we can assume that $\delta = \delta(n)$ approaches zero arbitrarily slowly as $n \to \infty$. Thus, by the induction hypothesis, we have
\begin{equation}\label{condd} 
\Ex_p\big( d_\P(v) \big) \; = \; \Ex_p\Big( \big| \Gamma_G(A \cap L_j, B^{(r-1)}(d-1,\ell+1,\eps) \log m, v) \big| \Big) \; \le \; \delta
\end{equation}
for any $v \in V(G_j)$, if $n$ (and therefore $m \ge \log_{(r)} n$) is sufficiently large. Moreover, we have $|S| \ge m \ge \log_{(r)} n$, so $\alpha^4 |S|^{\eps'} \to \infty$ as $n \to \infty$ if $\alpha(n) \to 0$ sufficiently slowly. 

By~\eqref{condc} and~\eqref{condd}, it follows that conditions $(c)$ and $(d)$ of Lemma~\ref{CClemma} are satisfied (for $\eps'$ and $\delta = \delta(\alpha,\eps')$ as above). Condition $(a)$ is satisfied by construction. Condition $(b)$ follows because if $\x \sim \y$ and $\x' \sim \y'$, and there are no bad edges, then either all four points are in the same internally spanned component with diameter at most $B \log n$, or they are in different components of $\{A\}(j)$. Thus, if $\x \not\sim \x'$, then the events $\x \sim \y$ and $\x' \sim \y'$ must occur disjointly, and so we can apply the van den Berg-Kesten Lemma.

Recall that $X(\P)$ denotes the event that there is a connected path across $G_\P$, and note that if $R$ is internally spanned by $A$, then, by the subclaim, the event $X(\P)$ holds. Thus, by Lemma~\ref{CClemma}, we have
$$P_p(R) \; \le \; \Pr\big( X(\P) \big) \; \le \; \alpha^{\lfloor m/k \rfloor} (m+k)^{d+\ell} $$
as required, and the claim follows.
\end{proof}

We shall now use the claim to prove the theorem for $r$. Indeed, suppose that $\diam([A]) \ge \log n$. By Lemma~\ref{smallcompt}, there exists an internally filled, connected set $X$ with
$$\frac{\log n - 1}{2} \; \le \; \diam(X) \; \le \; \log n - 1.$$
Let $R$ be the smallest rectangle containing $X$, and observe that $R$ is internally spanned by $A$, and that $\diam(R) = \diam(X) \le \log n$. Since there are at most $(n \log n)^d$ such rectangles, by the claim we have
$$\Pr_p\Big( \diam([A]) \ge \log n \Big) \; \le \; (n \log n)^d \cdot \alpha^{\log n / 3k} (\log n + k)^{d+\ell} \; \le \; n^{-dr},$$
if $n$ is sufficiently large, since $\alpha(n) \to 0$ as $n \to \infty$.

Finally, let $v \in V(G)$, and suppose that $w \in \Gamma_G(A, \log n, v)$. Then there exists an internally filled connected component $X \subset V(G)$ such that $v,w \in X$ and $\diam(X) \le \log n$, and hence there exists an internally spanned rectangle $R$ (the smallest rectangle containing $X$) such that $v,w \in R$ and $m := \diam(R) \le \log n$. 

There are at most $m^{2d}$ rectangles with diameter $m$ containing $v$, and each contains at most $m^d k^\ell$ vertices. It follows that
$$\Ex_p\Big( \big| \Gamma_G(A, \log n, v) \big| \Big) \; \le \; \sum_{m=1}^{\log n} m^{3d} k^\ell \cdot \alpha^{m / k} (m + k)^{d+\ell} \; = \; o(1),$$
since $\alpha(n) \to 0$ as $n \to \infty$. This completes the induction step, and hence the proof of Theorem~\ref{genthm}.
\end{proof}

This completes the proof of Theorem~\ref{sharp}, since the upper bound was proved in~\cite{d=r=3}, and the lower bound follows immediately from Theorem~\ref{genthm} in the case $\ell = 0$.

\section{Open problems}\label{Qsec}

In this section we shall present three different directions for future research into the bootstrap process on the grid $[n]^d$: extensions to higher dimensions ($d = d(n) \to \infty$), more general update rules, and further sharpening of the thresholds. See~\cite{BB,Maj,n^d,DCH,GHM} for some recent work on these questions.

\subsection{Higher dimensions}

We consider Theorem~\ref{sharp} to be an important step towards a much bigger goal: to determine $p_c([n]^d,r)$ for \emph{arbitrary} functions $n = n(t)$, $d = d(t)$ and $r = r(t)$ with $n+d \to \infty$. Despite much recent progress, $r$-neighbour bootstrap percolation on $[n]^d$ is still poorly understood for most such functions. 

Our understanding of the bootstrap process is most complete in the case $r = 2$, where we have sharp bounds in the case $d = O(1)$ (by Theorem~\ref{sharp}), and in the case $d \gg \log n$, where it was proved in~\cite{n^d} that
$$p_c\big( [n]^d,2 \big) \: = \: \Big( 4\lambda + o(1) \Big) \left( \frac{n}{n-1} \right)^2 \, \ds\frac{1}{d^2} \, 2^{-2\sqrt{d \log_2 n}},$$
as $d \to \infty$, where $\lambda \approx 1.166$ is the smallest positive root of the equation $\ds\sum_{k=0}^\infty \frac{(-1)^k \lambda^k}{2^{k^2-k} k!} = 0$. 

\begin{prob}
Determine $p_c([n]^d,2)$ for all functions $d(n)$ with $1 \ll d(n) = O(\log n)$. 
\end{prob}

We expect that our proof of Theorem~\ref{sharp} can be extended to slowly growing functions $d = d(n)$, and that $d = \Theta(\log n)$ will be the most challenging range. The growth of the the critical droplet is very different in the ranges $d = O(1)$ (where it grows in all directions at the same time), and $d \gg \log n$ (where it grows in only one direction at a time), and it  will be particularly interesting to see whether these are the only two possible (dominant) behaviours.

Due to some recent progress, we also know a significant amount about the process when $d = r$. Indeed, by Theorem~\ref{sharp} and the results of~\cite{Maj}, we have sharp bounds on $p_c([n]^d,d)$ when $d = O(1)$ and when $d \ge (\log \log n)^{2 + \eps}$. Looking from slightly further away, we have the following theorem, which is implied by the results of~\cite{Sch} and~\cite{Maj}.

\begin{thm}
Let $n = n(d)$. Then, as $d \to \infty$, we have
$$p_c([n]^d,d) \; = \; \left\{
\begin{array} {c @{\quad \textup{if} \quad}l}
o(1) & n \ge 2^{2^{\Ddots^2}}, \;\textup{ (a tower of 2s of height $d$)}\\[+1ex]
\ds\frac{1}{2} + o(1) & n \le 2^{2^{\sqrt{d / \log d}}}.
\end{array}\right.$$
\end{thm}

We have very little idea where (or how) the transition from 0 to $1/2$ occurs. 

\begin{prob}
Determine a function $n = n(d)$ (if one exists), such that 
$$0 \; < \;  \liminf_{d \to \infty} p_c([n]^d,d) \; \le \; \limsup_{d \to \infty} p_c([n]^d,d) \; < \; \frac{1}{2}.$$
\end{prob}

There are also much simpler questions to which we have no good answer. For example, the following conjecture was made in~\cite{n^d}.

\begin{conj}\label{rconj}
For $r$ fixed,
$$p_c([2]^d,r) \; = \; \exp\Big( - \Theta\left( d^{1 / 2^{r-1}} \right) \Big).$$
\end{conj}

We know of no non-trivial lower bound on this function when $r \ge 3$.

\subsection{More general models on $[n]^d$}

In~\cite{DCH}, Duminil-Copin and Holroyd introduced the following, much more general family of bootstrap percolation models. We say that $\n \subset \ZZ^d$ is a \emph{neighbourhood} if it is a finite, convex, symmetric set containing the origin $\0$. (Here, symmetric means that if $\x \in \n$ then $-\x \in \n$.) For $r \in \mathbb{N}$, define the bootstrap process on $[n]^d$ with parameters $(r,\n)$ by setting
$$A_{t+1} \; := \; A_t \cup \big\{v \in [n]^d : |(\n+ v) \cap A_t| \ge r \big\}$$
for each $t \in \N$, where $A_0$ is the set of vertices which are infected at time $0$. We define the closure $[A]$ of a set $A$, and the critical probability $p_c([n]^d,r,\n)$ as in the Introduction.

Depending on the shape of $\n$ and the value of $r$, these dynamics behave very differently. We say that they are \emph{critical} if, on the infinite grid $\ZZ^d$, any finite set generates a finite set, and no finite set can be the complement of a stable set (see~\cite{GG1,GG2} or~\cite{DCH} for more details). 

Critical models can be divided into two sub-families: balanced and unbalanced models. Given a set $S$, define $\iota_1(S)$ to be the maximal cardinality of a set of the form $L \cap S$ where $L$ is a line passing through $\0$. A model is \emph{balanced} if there exist two distinct lines $L$ and $L'$ passing through $\0$ such that $L \cap \n$ and $L' \cap \n$ both have cardinality $\iota_1(\n)$. Finally, let $\gamma_1 = r - (|\n| - \iota_1(\n))/2$.

The following theorem shows that, in two dimensions, there is a sharp threshold for $p_c([n]^d,r,\n)$ for all balanced, critical models.

\begin{thm}[Duminil-Copin and Holroyd~\cite{DCH}]\label{DCHthm}
Let $\n \subset \ZZ^2$ be a neighbourhood of $\0$, let $r \in \N$, and suppose that the bootstrap process on $[n]^2$ with parameters $(r,\n)$ is balanced and critical. Then there exists a constant $\Lambda \in (0, \infty)$ such that
$$p_c([n]^2,r,\n) \; = \; \left( \frac{\Lambda + o(1)}{\log n} \right)^{1/\gamma_1}$$
as $n \to \infty$.
\end{thm}

It is a challenging open problem to extend this result to higher dimensions, and to more general neighbourhoods and update rules.

\subsection{Sharper thresholds}

Finally, we note some recent progress, also in two dimensions, on the problem of proving even sharper thresholds for $p_c([n]^d,r)$. This question was first addressed by Gravner and Holroyd~\cite{GH1}, who were interested in explaining the surprising discrepancy between the rigorously proved result of Holroyd~\cite{Hol}, and the estimates of $p_c([n]^2,2)$ from simulations. They improved the upper bound, proving that
$$p_c([n]^2,2) \; \le \; \frac{\pi^2}{18 \log n} \,-\, \frac{c}{(\log n)^{3/2}}$$
for some $c > 0$, and showed also that the function $p_c([n]^2,2) \log n$ converges too slowly for the limit to be easily estimated. In~\cite{GH2}, they conjectured that their upper bound is close to being tight. This conjecture was proved recently by Gravner, Holroyd and Morris~\cite{GHM}.

\begin{thm}[Gravner, Holroyd and Morris~\cite{GHM}]\label{sharper}
There exist constants $C > 0$ and $c > 0$ such that
$$\frac{\pi^2}{18 \log n} \, - \, \frac{C(\log\log n)^3} {(\log n)^{3/2}} \; \le \; p_c([n]^2,2) \; \le \; \frac{\pi^2}{18 \log n} \, - \, \frac{c} {(\log n)^{3/2}}$$
for every $n \in \N$.
\end{thm}

Similarly tight bounds have also been proved for the hypercube when $r = 2$ and when $r = d/2$ (see~\cite{Maj,n^d}). By combining the techniques of this paper with those of ~\cite{GHM}, one might hope that sharper bounds could also be given on $p_c([n]^d,2)$. However, it is likely to be much harder to prove such results when $r \ge 3$.

\end{document}